\documentclass[english]{article}

\usepackage{amsmath}
\usepackage{amssymb}
\usepackage{graphicx}
\usepackage{epic,eepic}

\usepackage{color}

\usepackage{amstext,amsthm,amssymb,amsmath}
\usepackage{graphicx}
\usepackage{subfigure}
\usepackage{wrapfig}
\usepackage{fullpage}
\usepackage{color}
\usepackage{multirow}
\usepackage{tabulary}
\usepackage{booktabs}
\usepackage{enumerate}
\usepackage{eepic,epic}
\usepackage{epsfig,subfigure,epstopdf}
\usepackage{algorithm,algorithmicx,algpseudocode}

\newtheorem{theorem}{Theorem}[section]
\newtheorem{lemma}[theorem]{Lemma}

\newtheorem{remark}[theorem]{Remark}

\hfuzz=\maxdimen
\title{Generalized multiscale finite element methods for space-time
heterogeneous parabolic equations}

\author{Eric T. Chung\thanks{Department of Mathematics, The Chinese University of Hong Kong, Hong Kong SAR.
This research is partially supported by the Hong Kong RGC General Research Fund (Project number: 400411).}, \and
Yalchin Efendiev\thanks{Department of Mathematics, Texas A\&M University, College Station, TX; Numerical Porous Media SRI Center, King Abdullah University of Science and Technology (KAUST), Thuwal 23955-6900, Kingdom of Saudi Arabia},\and
Wing Tat Leung\thanks{Department of Mathematics, Texas A\&M University, College Station, TX.},\and
 Shuai Ye\thanks{Department of Mathematics, Texas A\&M University, College Station, TX.}
}
\date{}

\usepackage{babel}
\begin{document}

\maketitle
%==============================

\begin{abstract}

In this paper, we consider local multiscale model reduction for
problems with multiple scales in space and time. We developed our approaches within the framework of the Generalized Multiscale Finite Element Method (GMsFEM) using space-time coarse cells. The main idea of GMsFEM is to construct a local snapshot space and a local spectral decomposition in the snapshot space. Previous research in developing multiscale spaces within GMsFEM focused on constructing multiscale spaces and relevant ingredients in space only. In this paper, our main objective is to develop a multiscale
model reduction framework within GMsFEM that uses space-time coarse
cells. We construct space-time snapshot and offline spaces.
We compute these snapshot solutions by solving local problems. A complete snapshot space will use all possible
boundary conditions; however, this can be very expensive. We
propose using randomized boundary conditions and oversampling
(cf. \cite{randomized2014}). We construct the local spectral decomposition
based on our analysis, as presented in the paper.
We present numerical results to confirm our theoretical findings and
to show that using our proposed approaches, we can obtain an accurate
solution with low dimensional coarse spaces. We discuss using online basis functions constructed in the online stage
and using the residual information. Online basis functions use
global information via the residual
and provide fast convergence to the exact solution
provided a sufficient number of offline basis functions.
We present numerical studies for our proposed online procedures.
We remark that the proposed method is
a significant extension compared to
existing methods, which use coarse cells in space only because
of (1) the parabolic nature of cell solutions, (2) extra degrees of freedom
associated with space-time cells, and (3) local boundary conditions in space-time cells.

\end{abstract}

\section{Introduction}

%==============================

%General multiscale space and time. Challenges and methods.

Many multiscale processes vary over multiple space and time scales.
These space and time scales are often tightly coupled. For example,
flow processes in porous media
 can occur on multiple time scales over multiple spatial scales.
Moreover, these scales can be non-separable.
Reduced-order models for these problems require simultaneously treating
spatial and temporal scales. Many previous approaches only handle spatial scales and spatial heterogeneities. These approaches
have limitations when temporal heterogeneities arise.
In this paper, we discuss a class of multiscale methods for handling
space and time scales.

%Mention homogenization and upscaling. Numerical homogenization.
Some well-known approaches for handling {\it separable}
spatial and temporal scales
are homogenization techniques \cite{jikov2012homogenization,pankov2013g,pavliotis2008multiscale,efendiev2005homogenization}. In these methods, one solves local problems
in space and time. To give an example, we consider
a well-known case of the parabolic equation
\begin{equation}
\label{eq:parabolic_epsilon}
\begin{split}
\frac{\partial}{\partial t}u-\text{div}(\kappa(x,x/\epsilon^\alpha,t,t/\epsilon^\beta)\nabla u)  = f,
\end{split}
\end{equation}
subject to smooth initial and boundary conditions.  
Here, $\epsilon$ is a small scale, and the spatial scale is $\epsilon^\alpha$, and
the temporal scale is $\epsilon^\beta$.
One can show
that (e.g., \cite{jikov2012homogenization,pankov2013g}),
the homogenized equation has the same form
as (\ref{eq:parabolic_epsilon}), but with the smooth coefficients
$\kappa^*(x,t)$. One can compute the coefficients using the solutions of local space-time parabolic equations in the  periodic cell.
This localization is possible thanks to the scale separation.
The local problems may or may not include time-dependent
derivatives depending on the interplay between $\alpha$ and $\beta$
since the cell problems are independent of $\epsilon$.
One can extend this homogenization procedure to numerical homogenization
type methods
\cite{ming2007analysis,abdulle2014finite,efendiev2004numerical, fish2004space}, where
one solves the local parabolic equations in each coarse block and
in each coarse time step. To compute the effective property, 
one averages the solutions of the local problems.
These approaches work well in the scale separation cases, but do not
provide accurate approximations when there is no scale separation.

Previous researchers developed a number of multiscale methods for solving
space-time multiscale problems in the absence of scale separation.
These approaches use Multiscale Finite Element Methods \cite{hw97,eh09, kunze2012adaptive, efendiev2004numerical},
where one computes multiscale space-time basis functions, variational multiscale
methods \cite{hughes1996space, hughes1996space_1}, and other approaches 
\cite{takizawa2011multiscale, tezduyar1992computation,nguyen1984space,masud1997space}
that are 
developed for stabilization.
In \cite{owhadi2007homogenization}, Owhadi and
Zhang proposed a novel approach that uses global space-time information
in computing multiscale basis functions. All these approaches use only a limited
number of basis functions (one basis function)
in each coarse block. 
We note that there has been a large body of works in space-time finite element methods. 
In this paper, our objective is to develop a
general approach that can systematically construct multiscale basis functions, and provide
analysis for multiscale high-contrast problems.

Our approaches use the Generalized Multiscale Finite Element Method (GMsFEM)
Framework and develop
a systematic approach for identifying multiscale basis functions.
The GMsFEM is a generalization of MsFEM, proposed by Hou and Wu \cite{hw97}.
The main idea of the GMsFEM is to construct multiscale basis functions
by constructing snapshots spaces and performing local spectral decomposition
in the snapshot spaces \cite{ egh12, Chung_adaptive14, chan2015adaptive, Efen_GVass_11, ce09,chung2015generalizedperforated, Efen_GVass_11, Efendiev_GLW_ESAIM_12, Efen_GVass_11,chung2013sub,chung2015residual,chung2014generalized,chung2015online, bush2014application}. The choice of the snapshot spaces and
the local spectral decomposition is important for converging the
resulting approach. We choose the snapshot spaces such that it can approximate
the local solution space, while typically deriving local spectral decomposition from the analysis.

Previous approaches in developing multiscale
spaces within GMsFEM focused
on constructing
 multiscale spaces and relevant ingredients in space only.
The proposed method is
a significant extension compared to
existing methods, which use coarse cells in space only because
of (1) the parabolic nature of cell solutions, (2) extra degrees of freedom
associated with space-time cells,
and (3) local boundary conditions in space-time cells.
In our approach, we construct snapshot spaces in space-time
local domains. We construct the snapshot solutions
by solving local problems.
We can construct a complete snapshot space by taking all
possible boundary conditions; however, this can result to very high
computational cost. For this reason, we
use randomized boundary conditions for local snapshot vectors by
solving
parabolic equations subject to random boundary and initial conditions.
We compute only a few more than the number of basis functions
needed. Computing multiscale basis functions employs local
spectral problems. These local spectral problems are in space-time domain.
Using space-time eigenvalue problems controls
the errors associated with $\partial u /\partial t$. We discuss
several choices for local spectral problems and present a convergence
analysis of the method.

In the paper, we present several numerical examples. We consider
the numerical tests with the conductivities that contain high contrast
and these high
conductivity regions move in time. These are  challenging examples
since the high-conductivity heterogeneities vary significantly during
one coarse-grid time interval. If
only using spatial multiscale basis functions, one will need a very large dimensional coarse space.
In our numerical
results, we use oversampling and randomized snapshots. Our results
show that one can achieve a small error by selecting a few multiscale
basis functions. The numerical results confirm our convergence analysis.

In the paper, we also discuss online multiscale basis functions.
In \cite{chung2015residual,chung2015online}, we present an online
procedure for {\it time-independent} problems.
The main idea of online multiscale basis functions is to use the residual
information and construct new multiscale basis functions adaptively.
We would like to choose a number of offline basis functions such that with
only 1-2 online iterations, we can substantially reduce the error.
This requires a sufficient number of online basis functions, with the online basis function construction typically derived by the analysis.
In this paper, we present a possible online construction and
show numerical results. Based on our previous results for
{\it time-independent} problems, we show that one needs a sufficient
number of offline basis functions to reduce the error substantially.
In our numerical results, we observe a similar phenomena, i.e.,
the error decreases rapidly in 1-2 online iterations.
We plan to investigate the convergence of the online procedure in our future work.

We organize the paper as follow. In Section \ref{sect:GMsFEM}, we present the underlying problem, the concepts of coarse and fine grids, the motivation of space-time approach, and the space-time GMsFEM framework. In Section \ref{sect:analysis}, we present the convergence analysis for our proposed method. In Section \ref{sect:online}, we present the new enrichment procedure of computing online multiscale basis functions. We present numerical results for offline GMsFEM and online GMsFEM in Section \ref{sect:NR1} and Section \ref{sect:NR2}, separately. In Section \ref{sect:conclusion}, we draw conclusions.

\section{Space-time GMsFEM}\label{sect:GMsFEM}

\subsection{Preliminaries and motivation}

Let $\Omega$ be a bounded domain in $\mathbb{R}^2$ with a Lipschitz boundary $\partial \Omega$, and $[0,T]$ $(T>0)$ be a time interval. In this paper, we consider the following parabolic differential equation
\begin{equation}\label{eq:para PDE}
\begin{split}
\frac{\partial}{\partial t}u-\text{div}(\kappa(x,t)\nabla u) & = f \quad\qquad \text{in }\Omega\times(0,T), \\
u & = 0 \quad\qquad \text{on }\partial\Omega\times(0,T),\\
u(x,0) & = \beta(x) \quad\;\; \text{in }\Omega,
\end{split}
\end{equation}
where $\kappa(x,t)$ is a time dependent heterogeneous media
(for example, a time dependent high-contrast permeability field),
$f$ is a given source function, $\beta(x)$ is the initial condition.
Our main objective is to develop
space-time multiscale model reduction within GMsFEM and we use
the time-dependent parabolic equation as an example.
The proposed methods can be used for other models that require
space-time multiscale model reduction.

We will introduce the space-time generalized multiscale finite element method in this section. The method follows the space-time finite element framework, where the time dependent multiscale basis functions are constructed on the coarse grid. Therefore, compared with the time independent basis structure, it gives a more efficient numerical solver for the parabolic problem in complicated media.

Before introducing our method, we need to define the mesh of the domain first. Let $\mathcal{T}^{h}$ be a partition of the domain $\Omega$ into fine finite elements where $h>0$ is the fine mesh size. Then we form a coarse partition $\mathcal{T}^{H}$ of the domain $\Omega$ such that every element in $\mathcal{T}^{H}$ is a union of connected fine-mesh
grid blocks, that is, $\forall K_j\in\mathcal{T}^{H}$, $K_j=\cup_{F\in I_j}F$ for some $I_j\subset\mathcal{T}^{h}$. The set $\mathcal{T}^{H}$ is called the coarse grid and the elements of $\mathcal{T}^{H}$ are called coarse elements. Moreover, $H>0$ is the coarse mesh size. In this paper, we consider rectangular coarse elements for the ease of discussions and illustrations. The methodology presented can be easily extended to coarse elements with more general geometries. Let $\{x_i\}_{i=1}^{N_c}$ be the set of nodes in the coarse grid $\mathcal{T}^{H}$ (or coarse nodes for short), where $N_c$ is the number of coarse nodes. We denote the neighborhood of the node $x_i$ by
\[
\omega_i = \bigcup\{K_j\in\mathcal{T}^{H}: x_i\in\overline{K_j} \}.
\]
Notice that $\omega_i$ is the union of all coarse elements $K_j\in\mathcal{T}^{H}$ sharing the coarse node $x_i$. An illustration of the above definition is shown in Figure \ref{Fig:plotnhood}. Next, let $\mathcal{T}^{T}=\{(T_{n-1},T_{n})|1\leq n\leq N\}$
be a coarse partition of $(0,T)$ where
\[
0=T_{0}<T_{1}<T_{2}<\cdots<T_{N}=T
\]
and we define a fine partition of $(0,T)$, $\mathcal{T}^{t}$ by refining
the partition $\mathcal{T}^{T}$.

\begin{figure}[H]
\begin{minipage}[t]{0.45\linewidth}
\centering
\includegraphics[width=\columnwidth]{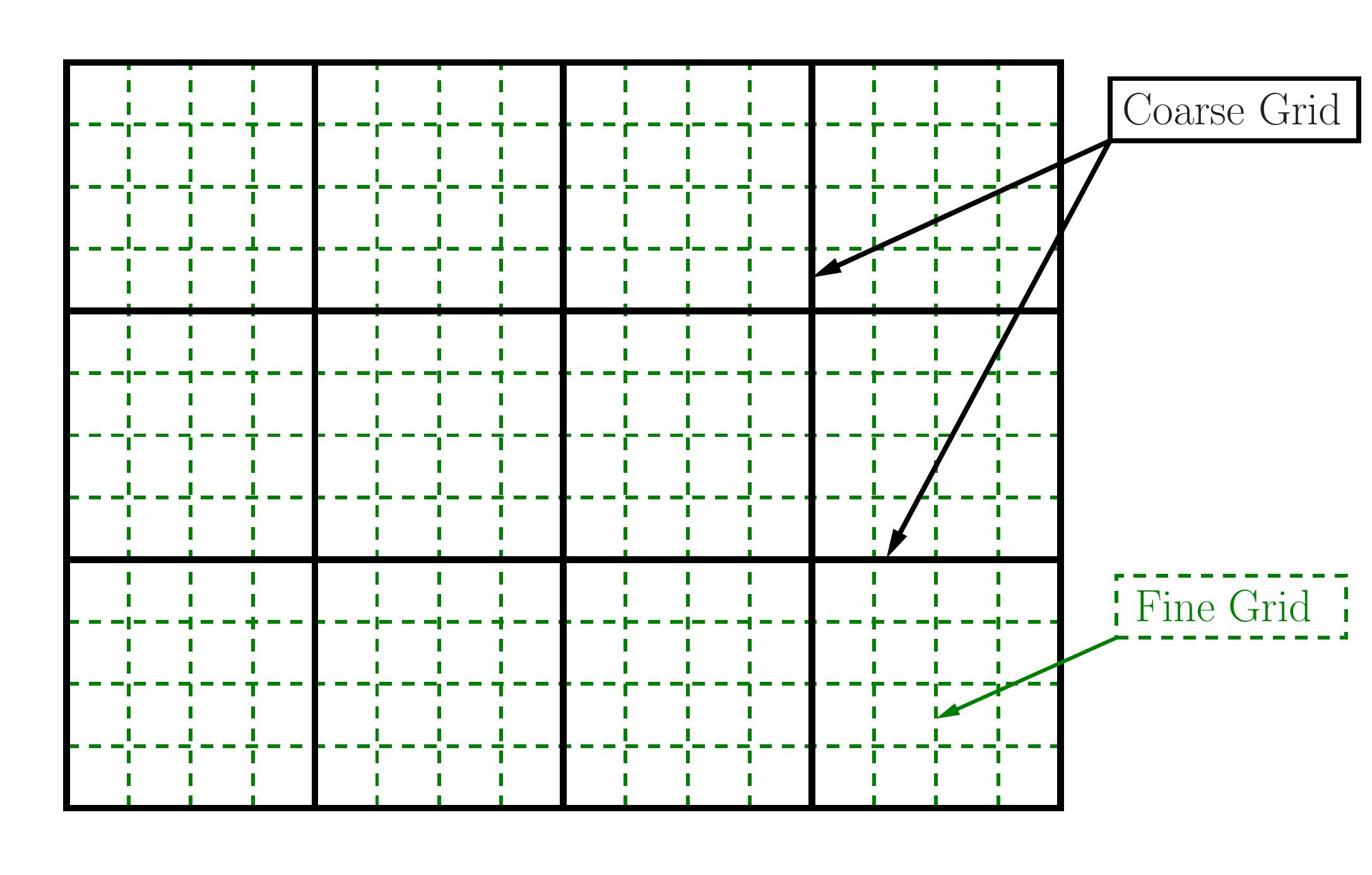}
\end{minipage}
\begin{minipage}[t]{0.45\linewidth}
\centering
\includegraphics[width=\columnwidth]{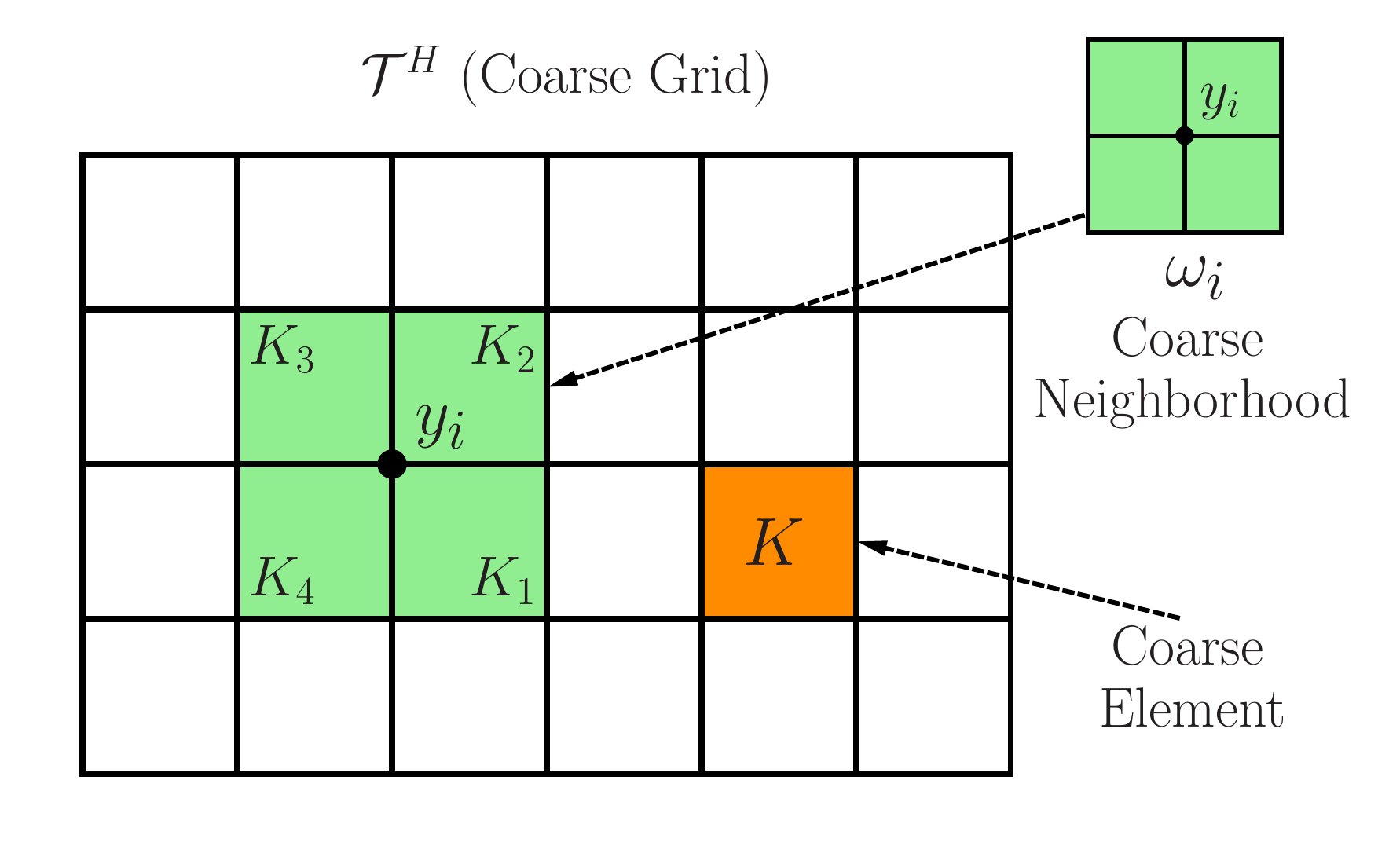}
\end{minipage}
\caption{Left: an illustration of fine and coarse grids. Right: an illustration of a coarse neighborhood and a coarse element.}
\label{Fig:plotnhood}
\end{figure}

%We remark that one use of fine space grid and fine time grid is for the computation of a fine-scale solution, which is used as a reference solution for
%comparison purposes.
To fix the notations, we will use the standard conforming piecewise linear finite element method for the computation of the fine-scale solution.
One can use discontinuous Galerkin coupling also \cite{eglmsMSDG,cel14,eric-2012}.
Specifically, we define the finite element space $V_h$ with respect to $\mathcal{T}^{h}\times(0,T)$ as
\begin{eqnarray*}
V_{h} &=& \{v\in L^{2}((0,T);C^{0}(\Omega))\; | \; v=\phi(x)\psi(t) \text{ where }\phi|_{K}\in Q_{1}(K)\;\forall K\in\mathcal{T}^{h}, \; \psi|_{\tau}\in C^{0}(\tau)\;\forall \tau\in\mathcal{T}^{T} \\
& & \text{ and }\psi|_{\tau}\in P_{1}(\tau)\;\forall \tau\in\mathcal{T}^{t}\},
\end{eqnarray*}
then the fine-scale solution $u_h \in V_h$ is obtained by solving the following variational problem
\begin{equation}\label{eq:fine problem}
\int_{0}^{T}\int_{\Omega}\cfrac{\partial u_{h}}{\partial t}v+\int_{0}^{T}\int_{\Omega}\kappa\nabla u_{h}\cdot\nabla v+\sum_{n=0}^{N-1}\int_{\Omega}[u_{h}(x,T_{n})]v(x,T_{n}^{+})
=\int_{0}^{T}\int_{\Omega}fv+\int_{\Omega}\beta(x)v(x,T_{0}^{+}),\;\forall v\in V_{h},
\end{equation}
where $[\cdot]$ is the jump operator such that
\[
\begin{cases}
[u_{h}(x,T_{n})]=u_{h}(x,T_{n}^{+})-u_{h}(x,T_{n}^{-}) & \text{ for }n\geq1,\\
{}[u_{h}(x,T_{n})]=u_{h}(x,T_{0}^{+}) & \text{ for }n=0.
\end{cases}
\]
We assume that the fine mesh size $h$ is small enough so that the fine-scale solution $u_h$ is close enough to the exact solution. The purpose of this paper is to find a multiscale solution $u_{H}$ that is a good approximation of the fine-scale solution $u_h$.

Now we present the general idea of GMsFEM.
% We will consider the continuous Galerkin (CG) formulation, which has a similar form as the fine-scale problem (\ref{eq:fine problem}).
We will use the space-time finite element method to solve problem (\ref{eq:para PDE}) on the coarse grid. That is, we find $u_{H}\in V_{H}$ such that
\begin{equation}\label{eq: space-time FEM}
\int_{0}^{T}\int_{\Omega}\cfrac{\partial u_{H}}{\partial t}v+\int_{0}^{T}\int_{\Omega}\kappa\nabla u_{H}\cdot\nabla v+\sum_{n=0}^{N-1}\int_{\Omega}[u_{H}(x,T_{n})]v(x,T_{n}^{+})
=\int_{0}^{T}\int_{\Omega}fv+\int_{\Omega}\beta(x)v(x,T_{0}^{+}),\;\forall v\in V_{H},
\end{equation}
where $V_{H}$ is the multiscale finite element space which will be introduced in the following subsections.

The computational cost for solving the equation (\ref{eq: space-time FEM})
is huge since we need to compute the solution $u_{H}$ in the whole
time interval $(0,T)$ at one time. In fact, if we assume the solution
space $V_{H}$ is a direct sum of the spaces only containing the functions
defined on one single coarse
time interval $(T_{n-1},T_{n})$, we can decompose
the problem (\ref{eq: space-time FEM}) into a sequence of problems
and find the solution $u_{H}$ in each time interval sequentially.
Our coarse space will be constructed in each time interval and
we will have
\[
V_{H}=\oplus_{n=1}^{N}V_{H}^{(n)},
\]
 where $V_{H}^{(n)}$ only contains the functions having zero values
in the time interval $(0,T)$ except $(T_{n-1},T_{n})$, namely
$\forall v\in V_{H}^{(n)},$
\[
v(\cdot,t)=0\text{ for }t\in(0,T)\backslash(T_{n-1},T_{n}).
\]
%\begin{assumption}\label{assumption: decouple the time interval}The
%solution space $V_{H}$ can be written in the following expression:
%\[
%V_{H}=\oplus_{n=1}^{N}V_{H}^{(n)},
%\]
% where $V_{H}^{(n)}$ only contains the functions having zero values
%in the time interval $(0,T)$ except $(T_{n-1},T_{n})$, namely
%$\forall v\in V_{H}^{(n)},$
%\[
%v(\cdot,t)=0\text{ for }t\in(0,T)\backslash(T_{n-1},T_{n}).
%\]
%\end{assumption}

%Using this assumption,
The equation (\ref{eq: space-time FEM}) can
be decomposed into the following problem: find $u_{H}^{(n)}\in V_{H}^{(n)}$
(where $V_{H}^{(n)}$ will be defined later)
satisfying
\begin{align} \label{eq:space-time FEM coarse decoupled}
 & \int_{T_{n-1}}^{T_{n}}\int_{\Omega}\cfrac{\partial u_{H}^{(n)}}{\partial t}v+\int_{T_{n-1}}^{T_{n}}\int_{\Omega}\kappa\nabla u_{H}^{(n)}\cdot\nabla v+\int_{\Omega}u_{H}^{(n)}(x,T_{n-1}^{+})v(x,T_{n-1}^{+})\nonumber \\
= & \int_{T_{n-1}}^{T_{n}}\int_{\Omega}fv+\int_{\Omega}g_{H}^{(n)}(x)v(x,T_{n-1}^{+}),\;\forall v\in V_{H}^{(n)},
\end{align}
where
\[
g_{H}^{(n)}(\cdot)=\begin{cases}
u_{H}^{(n-1)}(\cdot,T_{n-1}^{-}) & \text{ for }n\geq1,\\
\beta(\cdot) & \text{ for }n=0.
\end{cases}
\]
Then, the solution $u_{H}$ of the problem (\ref{eq: space-time FEM})
is the direct sum of all these $u_{H}^{(n)}$'s, that is $u_{H}=\oplus_{n=1}^{N}u_{H}^{(n)}$.

Next, we motivate the use of space-time multiscale basis functions
by comparing it to space multiscale basis functions.
In particular, we discuss the savings in the reduced models when space-time
multiscale basis functions are used compared to space multiscale
basis functions.
We denote $\{t_{n1},\cdot\cdot\cdot,t_{np}\}$ are $p$ fine time steps
in $(T_{n-1},T_n)$. When we construct space-time multiscale basis functions,
the solution can be represented as $u_H^{(n)} = \sum_{l,i} c_{l,i} \psi_l^{\omega_i}(x,t)$
in the interval  $(T_{n-1},T_n)$. In this case, the number of coefficients
$c_{l,i}$ is related to the size of the reduced system in space-time interval.
On the other hand, if we use only space multiscale basis functions,
we need to construct these multiscale basis functions at each
fine time instant $t_{nj}$, denoted by $\psi_{l}^{\omega_i}(x,t_{nj})$.
The solution $u_H$ spanned by these basis functions will have a much
larger dimension because each time instant is represented by
multiscale basis functions. Thus, performing space-time multiscale
model reduction can provide a substantial CPU savings.

In the next, we will discuss space-time multiscale basis functions.
First, we will construct multiscale basis functions in the offline mode
without using the residual. Next, in Section \ref{sect:online}, we will discuss
online space-time multiscale basis construction.

\subsection{Construction of offline basis functions}
\subsubsection{Snapshot space}

Let $\omega$ be a given coarse neighborhood in space.
We omit the coarse node index to simplify the notations. The construction of the
offline basis functions on coarse time interval $(T_{n-1},T_n)$ starts
with a snapshot space $V_{\text{snap}}^{\omega}$ (or $V_{\text{snap}}^{\omega (n)}$). We also omit the coarse time index $(n)$ to simplify the notations. The snapshot space
$V_{\text{snap}}^{\omega}$ is a set of functions defined on $\omega$
and contains all or most necessary components of the fine-scale
solution restricted to $\omega$. A spectral problem is then solved
in the snapshot space to extract the dominant modes in the snapshot space.
These dominant modes are the offline basis functions and the resulting
reduced space is called the offline space. There are two choices of
 $V_{\text{snap}}^{\omega}$ that are commonly used.

The first choice is to use all possible fine-grid functions in
$\omega\times (T_{n-1},T_{n})$. This snapshot spaces provide accurate
approximation for the solution space; however,
this snapshot space can be very large.
The second choice for the snapshot spaces consists
of solving local problems for all possible boundary conditions.
In particular, we define $\psi_{j}$ as the solution of
\begin{equation}\label{eq:locSnap}
\begin{split}
&\frac{\partial}{\partial t} \psi_{j} -\text{div} (\kappa(x,t) \nabla \psi_{j})=0\ \ \text{in}\ \omega\times (T_{n-1},T_{n}), \\
&\psi_{j}(x,t)=\delta_j(x,t)\  \ \text{on} \ \ \partial \left( \omega\times (T_{n-1},T_{n}) \right).
\end{split}
\end{equation}
Here $\delta_j(x,t)$ is a fine-grid delta function and $\partial \left( \omega\times (T_{n-1},T_{n}) \right)$ denotes
the boundaries $t=T_{n-1}$ and on $\partial \omega\times (T_{n-1},T_{n})$. In general, the computations of these snapshots are expensive since in each local coarse neighborhood $\omega$, $O(M_n^{\partial\omega})$ number of local problems are required to be solved. Here, $M_n^{\partial\omega}$ is the number of fine grids on the boundaries $t=T_{n-1}$ and on $\partial \omega\times (T_{n-1},T_{n})$. A smaller yet accurate snapshot space is needed to build a more efficient multiscale method. We can take an advantage of randomized oversampling concepts \cite{calo2014randomized} and compute only a few snapshot vectors, which will reduce the computational cost remarkably while keeping required accuracy. Next, we introduce randomized snapshots.

Firstly, we introduce the notation for oversampled regions.
We denote by $\omega^{+}$ the oversampled space region of
$\omega \subset\omega^{+}$, defined by adding several fine- or coarse-grid
layers around $\omega$. Also, we define $(T_{n-1}^{*}, T_{n})$ as
the left-side oversampled time region for $(T_{n-1},T_{n})$. In the following,
we generate inexpensive snapshots using random boundary conditions on
the oversampled space-time region $\omega^{+}\times(T_{n-1}^{*},T_{n})$.
That is, instead of solving Equation (\ref{eq:locSnap}) for each fine boundary
node on $\partial \left( \omega\times (T_{n-1},T_{n}) \right)$, we solve
a small number of local problems imposed with random boundary conditions
\begin{equation*}
\begin{split}
&\frac{\partial}{\partial t} \psi_{j}^{+} -\text{div} (\kappa(x,t) \nabla \psi_{j}^{+})=0\ \ \text{in}\ \omega^{+}\times (T_{n-1}^{*},T_{n}), \\
&\psi_{j}^{+}(x,t)= r_l\  \ \text{on} \ \ \partial \left( \omega^{+}\times (T_{n-1}^{*},T_{n}) \right),
\end{split}
\end{equation*}
where $r_l$ are independent identically distributed (i.i.d.) standard Gaussian random vectors on the fine-grid nodes of the boundaries $t=T_{n-1}^{*}$ and on $\partial \omega^{+}\times (T_{n-1}^{*},T_{n})$.
%Then, we can obtain the local random snapshots $\psi_{i}$ on the target region $\omega\times (T_{n-1},T_{n})$ by restricting the solution of this local problem, $\psi_{i}^{+}$, onto $\omega\times (T_{n-1},T_{n})$. Thus, the local snapshot space on $\omega\times (T_{n-1},T_{n})$ is
%\[
%V_{\text{snap}}^{\omega} = \text{span}\{\psi_{i}(x,t) | i=1,\cdot\cdot\cdot, L^{\omega}+p_{\text{bf}}^{\omega}\},
%\]
%where $L^{\omega}$ is the number of local offline basis we want to construct in $\omega$ and $p_{\text{bf}}^{\omega}$ is the buffer number. Later on, we use the same same buffer number for all $\omega$'s and simply use the notation $p_{\text{\text{bf}}}$. With these snapshots, we follow the procedure in the following subsection to generate offline basis functions by using an auxiliary spectral decomposition.
Then the local snapshot space on $\omega^{+}\times (T_{n-1}^{*},T_{n})$ is
\[
V_{\text{snap}}^{\omega^{+}} = \text{span}\{\psi_{j}^{+}(x,t) | j=1,\cdot\cdot\cdot, L^{\omega}+p_{\text{bf}}^{\omega}\},
\]
where $L^{\omega}$ is the number of local offline basis we want to construct in $\omega$ and $p_{\text{bf}}^{\omega}$ is the buffer number. Later on, we use the same buffer number for all $\omega$'s and simply use the notation $p_{\text{bf}}$. In the following sections, if we specify one special coarse neighborhood $\omega_i$, we use the notation $L_i$ to denote the number of local offline basis. With these snapshots, we follow the procedure in the following subsection to generate offline basis functions by using an auxiliary spectral decomposition.

\subsubsection{Offline space}

To obtain the offline basis functions, we need to perform a space reduction by appropriate spectral problems. Motivated by our later convergence analysis, we adopt the following spectral problem on $\omega^{+}\times (T_{n-1},T_{n})$:

Find $(\phi,\lambda)\in V_{\text{snap}}^{\omega^{+}}\times\mathbb{R}$ such that
\begin{equation}\label{eq:eig-problem}
A_n(\phi,v) = \lambda S_n(\phi,v), \quad \forall v \in V_{\text{snap}}^{\omega^{+}},
\end{equation}
where the bilinear operators $A_n(\phi,v)$ and $S_n(\phi,v)$ are defined by
\begin{equation}
\begin{split}
A_n(\phi,v) &= \frac{1}{2} \left( \int_{\omega^{+}}\phi(x,T_{n})v(x,T_{n}) + \int_{\omega^{+}}\phi(x,T_{n-1})v(x,T_{n-1}) \right) + \int_{T_{n-1}}^{T_{n}}\int_{\omega^{+}}\kappa(x,t)\nabla\phi \cdot \nabla v, \\
S_n(\phi,v) &= \int_{\omega_{+}}\phi(x,T_{n-1})v(x,T_{n-1}) +
\int_{T_{n-1}}^{T_{n}}\int_{\omega^{+}}\widetilde{\kappa}^{+}(x,t)\phi v,
\end{split}
\end{equation}
where the weighted function $\widetilde{\kappa}^{+}(x,t)$ is defined by
\[
\widetilde{\kappa}^{+}(x,t) = \kappa(x,t)\sum_{i=1}^{N_c}|\nabla\chi_i^{+}|^2,
\]
$\{\chi_i^{+}\}_{i=1}^{N_c}$ is a partition of unity associated with the oversampled coarse neighborhoods $\{\omega_i^{+}\}_{i=1}^{N_c}$ and satisfies $|\nabla\chi_i^{+}|\geq|\nabla\chi_i|$ on $\omega_i$ where $\chi_i$ is the standard multiscale basis function for the coarse node $x_i$ (that is, with linear boundary conditions for cell problems). More precisely,
\begin{equation}\label{eq:POU}
\begin{split}
-\text{div}(\kappa(x,T_{n-1})\nabla\chi_i) &= 0, \quad \text{in }K\in\omega_i,\\
\chi_i &= g_i, \quad \text{on }\partial K,
\end{split}
\end{equation}
for all $K\in\omega_i$, where $g_i$ is a continuous function on $\partial K$ and is linear on each edge of $\partial K$.

We arrange the eigenvalues $\{\lambda_j^{\omega^{+}}|j=1,2,\cdot\cdot\cdot\,L^{\omega}+p_{\text{bf}}^{\omega}\}$ from (\ref{eq:eig-problem}) in the ascending order, and select the first $L^{\omega}$ eigenfunctions, which are corresponding to the first $L^{\omega}$ ordered eigenvalues, and denote them by $\{\Psi_1^{\omega^{+},\text{\text{off}}},\cdot\cdot\cdot, \Psi_{L^{\omega}}^{\omega^{+},\text{\text{off}}}\}$. Using these eigenfunctions, we can define
\[
\psi_j^{\omega^{+}}(x,t) = \sum_{k=1}^{L^{\omega}+p_{\text{bf}}^{\omega}} (\Psi_j^{\omega^{+},\text{off}})_k \psi_k^{+}(x,t), \qquad j=1,2,\cdot\cdot\cdot, L^{\omega},
\]
where $(\Psi_j^{\omega^{+},\text{off}})_k$ denotes the $k$-th component of $\Psi_j^{\omega^{+},\text{off}}$, and $\psi_k^{+}(x,t)$ is the snapshot basis function computed on $\omega^{+}\times (T_{n-1}^{*},T_{n})$ as in the previous subsection. Then we can obtain the snapshots $\psi_{j}^{\omega}(x,t)$ on the target region $\omega\times (T_{n-1},T_{n})$ by restricting $\psi_j^{\omega^{+}}(x,t)$ onto $\omega\times (T_{n-1},T_{n})$. Finally, the offline basis functions on $\omega\times (T_{n-1},T_{n})$ are defined by $\phi_j^{\omega}(x,t) = \chi\psi_j^{\omega}(x,t)$, where $\chi$ is the standard multiscale basis function from (\ref{eq:POU}) for a generic coarse neighborhood $\omega$. We also define the local offline space on $\omega\times (T_{n-1},T_{n})$ as
\[
V_{\text{off}}^{\omega} = \text{span}\{\phi_{j}^{\omega}(x,t) | j=1,\cdot\cdot\cdot, L^{\omega} \}.
\]
Note that one can take $V_H^{(n)}$ in (\ref{eq:space-time FEM coarse decoupled}) as $V_H^{(n)} = V_{\text{off}}^{(n)} =  \text{span}\{\phi_{j}^{\omega_i}(x,t) | 1\leq i\leq N_c, 1\leq j\leq L_{i} \}$. As a result, $V_{H} = V_{\text{off}}= \oplus_{n=1}^{N}V_{H}^{(n)}$.

\begin{remark}
For the convenience of convergence analysis in Section \ref{sect:analysis}, we also denote by $\{\Psi_1^{\omega^{+},\text{\text{off}}},\cdot\cdot\cdot, \Psi_{L^{\omega}+p_{\text{bf}}^{\omega}}^{\omega^{+},\text{\text{off}}}\}$ all the eigenfunctions from (\ref{eq:eig-problem}) corresponding to the ordered eigenvalues, and define
\[
\psi_j^{\omega^{+}}(x,t) = \sum_{k=1}^{L^{\omega}+p_{\text{bf}}^{\omega}} (\Psi_j^{\omega^{+},\text{off}})_k \psi_k^{+}(x,t), \qquad j=1,2,\cdot\cdot\cdot, L^{\omega}+p_{\text{bf}}^{\omega}.
\]
We note that the snapshot space on $\omega^{+}\times (T_{n-1}^{*},T_{n})$ can be rewritten as
\[
V_{\text{snap}}^{\omega^{+}} = \text{span}\{\psi_j^{\omega^{+}}(x,t) | j=1,\cdot\cdot\cdot, L^{\omega}+p_{\text{bf}}^{\omega}\},
\]
and the snapshot space on $\omega\times (T_{n-1},T_{n})$ can be written as
\[
V_{\text{snap}}^{\omega} = \text{span}\{\psi_j^{\omega}(x,t) | j=1,\cdot\cdot\cdot, L^{\omega}+p_{\text{bf}}^{\omega}\},
\]
where each $\psi_j^{\omega}(x,t)$ is the restriction of $\psi_j^{\omega^{+}}(x,t)$ onto $\omega\times (T_{n-1},T_{n})$. By collecting all local snapshot spaces on each $\omega\times (T_{n-1},T_{n})$, we can obtain the snapshot space $V_{\text{snap}}^{(n)}$ on $\Omega\times (T_{n-1},T_{n})$.

The offline space can be rewritten as
\[
V_{\text{off}}^{\omega} = \text{span}\{\chi\psi_{j}^{\omega}(x,t) | j\leq L^{\omega} \}.
\]
\end{remark}

\begin{remark}
One can use a more general spectral problem in (\ref{eq:eig-problem}) with
\begin{equation}
\begin{split}
A(\phi,v) =& \frac{1}{2} \left( \int_{\omega}\phi(x,T_{n})v(x,T_{n}) + \int_{\omega}\phi(x,T_{n-1})v(x,T_{n-1}) \right) + \int_{T_{n-1}}^{T_{n}}\int_{\omega}\kappa(x,t)\nabla\phi \cdot \nabla v \\
 & + \int_{T_{n-1}}^{T_{n}}\int_{\omega}\kappa(x,t)(z_{\phi}z_{v}+\nabla z_{\phi}\cdot\nabla z_{v}),\\
S(\phi,v) =& \int_{\omega}\phi(x,T_{n-1})v(x,T_{n-1}) +
\int_{T_{n-1}}^{T_{n}}\int_{\omega}\widetilde{\kappa}(x,t)\phi v
+\int_{T_{n-1}}^{T_{n}}\int_{\omega}\kappa|\nabla\chi|^{2}z_{\phi}z_{v},
\end{split}
\end{equation}
where for any $w \in V_{\text{snap}}^{\omega}$, $z_w$ satisfies
\[
-z_{w}(x,t)+\nabla\cdot(\kappa(x,t)\nabla z_{w}(x,t))=\chi \frac{\partial w}{\partial t}, \;\quad \forall t\in(T_{n-1},T_{n}).
\]
With this spectral problem, one can simplify the proof presented in
Section \ref{sect:analysis}. However, the numerical implementation of this local spectral problem is more complicated.
\end{remark}

\section{Convergence analysis}\label{sect:analysis}

In this section, we will analyze the convergence of our proposed method. To start, we firstly define two norms that are used in the analysis. We define $\|\cdot\|_{V^{(n)}}^{2}$ and $\|\cdot\|_{W^{(n)}}^{2}$ by
\begin{align*}
\|u\|_{V^{(n)}}^{2} & =\int_{T_{n-1}}^{T_{n}}\int_{\Omega}\kappa|\nabla u|^{2}+\cfrac{1}{2}\int_{\Omega}u^{2}(x,T_{n}^{-})+\cfrac{1}{2}\int_{\Omega}u^{2}(x,T_{n-1}^{+}),\\
\|u\|_{W^{(n)}}^{2} & =\|u\|_{V^{(n)}}^{2}+\int_{T_{n-1}}^{T_{n}}\|u_{t}(\cdot,t)\|_{H^{-1}(\kappa,\Omega)}^{2},
\end{align*}
where
\[
\|u\|_{H^{-1}_{(\kappa,\Omega)}} = \sup_{v\in H^1_0(\Omega)}\cfrac{\int_{\Omega} u v}{(\int_\Omega \kappa |\nabla v|^2)^{\frac{1}{2}}}.
\]

In the following, we will show the $V^{(n)}$-norm of the error $u_{h}-u_{H}$ can be bounded by the $W^{(n)}$-norm of the difference $u_{h}-w$ for any $w\in V_{H}^{(n)}$, where $u_{h}$ is the fine scale solution from Eqn.(\ref{eq:fine problem}), $u_{H}$ is the multiscale solution from Eqn.(\ref{eq:space-time FEM coarse decoupled}), and $V_{H}^{(n)}$ is the multiscale space defined in the previous section. The proof of this lemma will be presented in the Appendix.

\begin{lemma}\label{lem: cea lemma}
Let $u_{h}$ be the fine scale solution from Equation (\ref{eq:fine problem}), $u_{H}$ be the multiscale solution from Equation (\ref{eq:space-time FEM coarse decoupled}). We have the following estimate
\[
\|u_{h}-u_{H}\|_{V^{(n)}}^{2}\leq
\begin{cases}
C\|u_{h}-w\|_{W^{(n)}}^{2} & \mbox{for n}=1,\\
C(\|u_{h}-w\|_{W^{(n)}}^{2}+\|u_{h}-u_{H}\|_{V^{(n-1)}}^{2}) & \mbox{for n}>1,
\end{cases}
\]
for any $w\in V_{H}^{(n)}$. If we define the $V^{(0)}$-norm to be $0$, then we can write
\[
\|u_{h}-u_{H}\|_{V^{(n)}}^{2}\leq
C(\|u_{h}-w\|_{W^{(n)}}^{2}+\|u_{h}-u_{H}\|_{V^{(n-1)}}^{2}) \quad \mbox{for n}\geq 1,
\]
for any $w\in V_{H}^{(n)}$.
\end{lemma}

Therefore, to estimate the error of our multiscale solution, we only need to find a function $w$ in $V_{H}^{(n)}$ such that $\|u_{h}-w\|_{W^{(n)}}$is small. Except for Lemma \ref{lem: cea lemma}, we still need the following lemma to estimate $\|u_{h}-w\|_{W^{(n)}}.$

\begin{lemma}  \label{lem: Caccioppoli}
For any $v$ satisfying
\[
\frac{\partial}{\partial t}v-\text{div}(\kappa(x,t)\nabla v)=0\ \ \text{in}\ \ \omega_i\times(T_{n-1},T_{n}),
\]
we have
\begin{equation}
\begin{split}\int_{\omega_i}\chi_i^{2}v^{2}(x,T_{n}) + \int_{T_{n-1}}^{T_{n}}\int_{\omega_i}\kappa|\chi_i^{2}||\nabla v|^{2} \preceq \int_{\omega_i}\chi_i^{2}v^{2}(x,T_{n-1}) + \int_{T_{n-1}}^{T_{n}}\int_{\omega_i}\kappa|\nabla\chi_i|^{2}v^{2},
\end{split}
\end{equation}
where the notation $F \preceq G$ means $F \leq \mathcal{C}G$ with a constant $\mathcal{C}$ independent of the mesh, contrast and the functions involved.
\end{lemma}

Now, we are ready to prove our main result in this section.
%\marginpar{Please, state the assumptions in the proof.}
\begin{theorem}\label{main thm}

Let $u_{h}$ be the fine scale solution from Equation (\ref{eq:fine problem}), $u_{H}$ be the multiscale solution from Equation (\ref{eq:space-time FEM coarse decoupled}). Let $\tilde{u}_{h}=\text{argmin}_{v\in V_{\text{snap}}^{(n)}}\{\|u_{h}-v\|_{W^{(n)}}\}$ and we denote $\tilde{u}_{h}=\sum_{i}\chi_{i}\tilde{u}_{h,i}$ with $\tilde{u}_{h,i}=\sum_{j}c_{i,j}\psi_{j}^{\omega_i}$.
There holds
\[
\|u_{h}-u_{H}\|_{V^{(n)}}^{2}\preceq
M(DEF+1)\sum_{i}\left(\frac{1}{\lambda_{L_{i}+1}^{\omega_i^{+}}}\|\tilde{u}_{h,i}^{+}\|_{V^{(n)}(\omega_i^{+})}^{2}\right)
+\|u_{h}-\tilde{u}_{h}\|_{W^{(n)}}^{2}+\|u_{h}-u_{H}\|_{V^{(n-1)}}^{2},
\]
where

$M=\max_{K}\{M_{K}\}$ with $M_{K}$ is the number of coarse neighborhoods $\omega_{i}$'s which have nonempty intersection with $K$,

$D=\max\{D_{i}\}$ with $D_{i}=\sup_{v\in H_{0}^{1}(\Omega)}\cfrac{\int_{\omega_{i}}\kappa|\nabla\chi_{i}|^{2}v^{2}+\int_{\omega_{i}}\kappa\chi_{i}^{2}|\nabla v|^{2}}{\int_{\omega_{i}}\kappa|\nabla v|^{2}+\int_{\omega_{i}}\kappa v{}^{2}}$,

$E=\sup_{w\in H_{0}^{1}(\Omega)}\cfrac{\int_{\Omega}\kappa|\nabla w|^{2}+\int_{\Omega}\kappa w{}^{2}}{\int_{\Omega}\kappa|\nabla w|^{2}}$,

$F=\max\{F_{i}\}$ with $F_i = \cfrac{1}{\min_{x\in\omega_i}\{|\chi_i^{+}(x)|^2\}}$,

$\tilde{u}_{h,i}^{+}=\sum_{j}c_{i,j}\psi_{j}^{\omega_i^{+}}$ and the local norm $\|\cdot\|_{V^{(n)}(\omega_i^{+})}$ is defined by
\[
\|v\|_{V^{(n)}(\omega_i^{+})}^{2}=\int_{T_{n-1}}^{T_{n}}\int_{\omega_i^{+}}\kappa|\nabla v|^{2}+\frac{1}{2}\int_{\omega_i^{+}}v^{2}(x,T_{n}^{-})+\frac{1}{2}\int_{\omega_i^{+}}v^{2}(x,T_{n-1}^{+}).
\]

\end{theorem}

\begin{proof}
By Lemma \ref{lem: cea lemma},
\begin{equation}\label{eq:general estimate}
\|u_{h}-u_{H}\|_{V^{(n)}}^2 \preceq
\inf_{w\in V_H^{(n)}}\|u_{h}-w\|_{W^{(n)}}^2+\|u_{h}-u_{H}\|_{V^{(n-1)}}^{2}.
\end{equation}
Therefore, we need to estimate $\inf_{w\in V_{H}^{(n)}}\|u_{h}-w\|_{W^{(n)}}^2.$
Note that $\tilde{u}_{h}=\sum_{i}\chi_{i}\tilde{u}_{h,i}=\sum_{i}\sum_{j}c_{i,j}\chi_{i}\psi_{j}^{\omega_i}$.
Using this expression, we can define a projection of $\tilde{u}_{h}$
into $V_{H}^{(n)}$ by
\begin{equation*}
P(\tilde{u}_{h})=\sum_{i}\sum_{j\leq L_{i}}c_{i,j}\chi_{i}\psi_{j}^{\omega_i}.\label{eq:expression of projection of u_h}
\end{equation*}
Then
\begin{align}
\inf_{w\in V_{H}^{(n)}}\|u_{h}-w\|_{W^{(n)}}^2
&\leq \|u_{h}-P(\tilde{u}_{h})\|_{W^{(n)}}^2 \nonumber\\
&\leq \|u_{h}-\tilde{u}_{h}\|_{W^{(n)}}^2 + \|\tilde{u}_{h}-P(\tilde{u}_{h})\|_{W^{(n)}}^2. \label{eq:estimate inf}
\end{align}
We will estimate $\|\tilde{u}_{h}-P(\tilde{u}_{h})\|_{W^{(n)}}^2$. \\

By the definition of $\|\cdot\|_{W^{(n)}}$, we have
\begin{align*}
\|\tilde{u}_{h}-P(\tilde{u}_{h})\|_{W^{(n)}}^{2} & =\|\sum_{i}\chi_{i}(\tilde{u}_{h,i}-P(\tilde{u}_{h,i}))\|_{V^{(n)}}^{2}
+\int_{T_{n-1}}^{T_{n}}\|\cfrac{\partial(\sum_{i}\chi_{i}(\tilde{u}_{h,i}-P(\tilde{u}_{h,i})))}{\partial t}\|_{H^{-1}(\kappa,\Omega)}^{2}\;,
\end{align*}
where $\tilde{u}_{h,i}=\sum_{j}c_{i,j}\psi_{j}^{\omega_i}$ and $P(\tilde{u}_{h,i})=\sum_{j\leq L_{i}}c_{i,j}\psi_{j}^{\omega_i}$.
Let $e_{i}=\tilde{u}_{h,i}-P(\tilde{u}_{h,i})$, then
$\tilde{u}_{h}-P(\tilde{u}_{h})=\sum_{i}\chi_{i}e_{i}$.
Therefore,
\begin{equation}\label{eq:W-norm estimate}
\|\tilde{u}_{h}-P(\tilde{u}_{h})\|_{W^{(n)}}^{2}=\|\sum_{i}\chi_{i}e_{i}\|_{V^{(n)}}^{2}
+\int_{T_{n-1}}^{T_{n}}\|\cfrac{\partial(\sum_{i}\chi_{i}e_{i})}{\partial t}\|_{H^{-1}(\kappa,\Omega)}^{2}.
\end{equation}
In the following, we will estimate the two terms on the right hand side of (\ref{eq:W-norm estimate}), separately. Then the proof is done.\\

First, we estimate the term $\|\sum_{i}\chi_{i}e_{i}\|_{V^{(n)}}^{2}$. We define the local norm $\|\cdot\|_{V^{(n)}(K)}$ by
\[
\|v\|_{V^{(n)}(K)}^{2}=\int_{T_{n-1}}^{T_{n}}\int_{K}\kappa|\nabla v|^{2} +\frac{1}{2}\int_{K}v^{2}(x,T_{n}^{-})+\frac{1}{2}\int_{K}v^{2}(x,T_{n-1}^{+}).
\]
Then we have
\[
\|\sum_{i}\chi_{i}e_{i}\|_{V^{(n)}}^{2}\leq\sum_{K}\|\sum_{i}\chi_{i}e_{i}\|_{V^{(n)}(K)}^{2}.
\]
Moreover,
\begin{align*}
\|\sum_{i}\chi_{i}e_{i}\|_{V^{(n)}(K)}^{2} & \leq M_{K}\sum_{i}\|\chi_{i}e_{i}\|_{V^{(n)}(K)}^{2},
\end{align*}
where $M_{K}$ is the number of coarse neighborhoods $\omega_{i}$'s which have
nonempty intersection with $K$. Therefore,
\begin{align}
\|\sum_{i}\chi_{i}e_{i}\|_{V^{(n)}}^{2} & \leq\sum_{K}M_{K}\sum_{i}\|\chi_{i}e_{i}\|_{V^{(n)}(K)}^{2}\nonumber\\
 & \leq M\sum_{i}\|\chi_{i}e_{i}\|_{V^{(n)}(\omega_{i})}^{2}, \label{eq:V-norm estimate 1}
\end{align}
where $M=\max_{K}\{M_{K}\}$. Now, we need to estimate the term $\|\chi_{i}e_{i}\|_{V^{(n)}(\omega_{i})}^{2}$. Since
$\nabla(\chi_{i}e_{i})=e_{i}\nabla\chi_{i}+\chi_{i}\nabla e_{i}$,
we obtain
\begin{eqnarray*}
\|\chi_{i}e_{i}\|_{V^{(n)}(\omega_{i})}^{2} & \leq & 2\int_{T_{n-1}}^{T_{n}}\int_{\omega_{i}}\kappa|\nabla\chi_i|^{2}e_{i}^{2}+2\int_{T_{n-1}}^{T_{n}}\int_{\omega_{i}}\kappa\chi_i^{2}|\nabla e_{i}|^{2}\\
 &  & +\cfrac{1}{2}\int_{\omega_{i}}\chi_i^{2}e_{i}^{2}(x,T_{n}^{-})+\cfrac{1}{2}\int_{\omega_{i}}\chi_i^{2}e_{i}^{2}(x,T_{n-1}^{+}).\nonumber
\end{eqnarray*}
Using Lemma \ref{lem: Caccioppoli}, we have
\begin{align*}
\|\chi_{i}e_{i}\|_{V^{(n)}(\omega_{i})}^{2}
& \preceq
\int_{T_{n-1}}^{T_{n}}\int_{\omega_{i}}\kappa|\nabla\chi_i|^{2}e_{i}^{2}
+\int_{\omega_{i}}\chi_i^{2}e_{i}^{2}(x,T_{n-1}^{+})\\
& \preceq
\int_{T_{n-1}}^{T_{n}}\int_{\omega_{i}}\kappa|\nabla\chi_i|^{2}e_{i}^{2}
+\int_{\omega_{i}}e_{i}^{2}(x,T_{n-1}^{+}).
\end{align*}
Now we introduce notations in $\omega_i^{+}$ and denote $e_{i}^{+}=\tilde{u}_{h,i}^{+}-P(\tilde{u}_{h,i}^{+})$,
where $\tilde{u}_{h,i}^{+}=\sum_{j}c_{i,j}\psi_{j}^{\omega_i^{+}}$ and $P(\tilde{u}_{h,i}^{+})=\sum_{j\leq L_{i}}c_{i,j}\psi_{j}^{\omega_i^{+}}$.
It is obvious that $\tilde{u}_{h,i}^{+}|_{\omega_i} = \tilde{u}_{h,i}$, $P(\tilde{u}_{h,i}^{+})|_{\omega_i} =P(\tilde{u}_{h,i})$ and $e_{i}^{+}|_{\omega_i} = e_{i}$. And there holds the following two inequalities,
\begin{equation}\label{eq:oversample estimate 1}
\int_{T_{n-1}}^{T_{n}}\int_{\omega_{i}}\kappa|\nabla\chi_i|^{2}e_{i}^{2} \leq \int_{T_{n-1}}^{T_{n}}\int_{\omega_{i}^{+}}\kappa|\nabla\chi_i^{+}|^{2}|e_{i}^{+}|^{2},
\end{equation}
and
\begin{equation}\label{eq:oversample estimate 2}
\int_{\omega_{i}}e_{i}^{2}(x,T_{n-1}^{+}) \leq \int_{\omega_i^{+}}|e_i^{+}(x,T_{n-1}^{+})|^{2}.
\end{equation}
Thus,
\begin{equation}\label{eq:V-norm estimate 2}
\|\chi_{i}e_{i}\|_{V^{(n)}(\omega_{i})}^{2}
\preceq
\int_{T_{n-1}}^{T_{n}}\int_{\omega_{i}^{+}}\kappa|\nabla\chi_i^{+}|^{2}|e_{i}^{+}|^{2}
+\int_{\omega_i^{+}}|e_i^{+}(x,T_{n-1}^{+})|^{2}.
\end{equation}
Substituting (\ref{eq:V-norm estimate 2}) into (\ref{eq:V-norm estimate 1}), we immediately obtain
\begin{equation}\label{eq:V-norm estimate 3}
\|\sum_{i}\chi_{i}e_{i}\|_{V^{(n)}}^{2} \preceq M\sum_{i}\left(\int_{T_{n-1}}^{T_{n}}\int_{\omega_{i}^{+}}\kappa|\nabla\chi_i^{+}|^{2}|e_{i}^{+}|^{2}
+\int_{\omega_i^{+}}|e_i^{+}(x,T_{n-1}^{+})|^{2}\right).
\end{equation}\\

Next, we will estimate the term $\int_{T_{n-1}}^{T_{n}}\|\frac{\partial(\sum_{i}\chi_{i}e_{i})}{\partial t}\|_{H^{-1}(\kappa,\Omega)}^{2}$.
By definition, we have
\begin{align}
\int_{T_{n-1}}^{T_{n}}\|\cfrac{\partial(\sum_{i}\chi_{i}e_{i})}{\partial t}\|_{H^{-1}(\kappa,\Omega)}^{2}
& =\int_{T_{n-1}}^{T_{n}}\sup_{w\in H_{0}^{1}(\Omega)}\frac{\left(\int_{\Omega}\sum_{i}\chi_{i}\frac{\partial e_{i}}{\partial t}w\right)^{2}}{\int_{\Omega}\kappa|\nabla w|^{2}}\nonumber \\
& \leq\int_{T_{n-1}}^{T_{n}}\sup_{w\in H_{0}^{1}(\Omega)}\frac{\left(\sum_{i}|\int_{\omega_i}\chi_{i}\frac{\partial e_{i}}{\partial t}w|\right)^{2}}{\int_{\Omega}\kappa|\nabla w|^{2}}.\label{eq:H^=00007B-1=00007D norm global to local}
\end{align}
Since $e_{i}$ satisfies the equation
\begin{equation*}
\frac{\partial}{\partial t}e_{i}-\text{div}(\kappa(x,t)\nabla e_{i})=0\text{ in }\omega_{i}\times(T_{n-1},T_{n}),\label{eq:loc pde}
\end{equation*}
we have
\begin{align*}
\int_{\omega_{i}}\chi_{i}\frac{\partial e_{i}}{\partial t}w & =-\int_{\omega_{i}}\kappa(x,t)\nabla e_{i}\cdot\nabla(\chi_{i}w)\\
 & =-\int_{\omega_{i}}\kappa(x,t)w\nabla e_{i}\cdot\nabla\chi_{i}-\int_{\omega_{i}}\kappa(x,t)\chi_{i}\nabla e_{i}\cdot\nabla w.
\end{align*}
Moreover,
\begin{align}
\left|\int_{\omega_{i}}\chi_{i}\frac{\partial e_{i}}{\partial t}w \right|
= & \left|-\int_{\omega_{i}}\kappa w\nabla e_{i}\cdot\nabla\chi_{i}-\int_{\omega_{i}}\kappa\chi_{i}\nabla e_{i}\cdot\nabla w\right|\nonumber \\
\leq & \left(\int_{\omega_{i}}\kappa w^{2}|\nabla\chi_{i}|^{2}\right)^{\frac{1}{2}}\left(\int_{\omega_{i}}\kappa|\nabla e_{i}|^{2}\right)^{\frac{1}{2}} + \left(\int_{\omega_{i}}\kappa\chi_{i}^{2}|\nabla w|^{2}\right)^{\frac{1}{2}}\left(\int_{\omega_{i}}\kappa|\nabla e_{i}|^{2}\right)^{\frac{1}{2}}\nonumber \\
\leq & 2\left(\int_{\omega_{i}}\kappa w^{2}|\nabla\chi_{i}|^{2}+\int_{\omega_{i}}\kappa\chi_{i}^{2}|\nabla w|^{2}\right)^{\frac{1}{2}}\left(\int_{\omega_{i}}\kappa|\nabla e_{i}|^{2}\right)^{\frac{1}{2}}.\label{eq:estimate loc H^=00007B-1=00007D norm 1}
\end{align}
Let
\[
D_{i}=\sup_{v\in H_{0}^{1}(\Omega)}\cfrac{\int_{\omega_{i}}\kappa|\nabla\chi_{i}|^{2}v^{2}+\int_{\omega_{i}}\kappa\chi_{i}^{2}|\nabla v|^{2}}{\int_{\omega_{i}}\kappa|\nabla v|^{2}+\int_{\omega_{i}}\kappa v{}^{2}}.
\]
From (\ref{eq:estimate loc H^=00007B-1=00007D norm 1}), we obtain
\[
\left|\int_{\omega_{i}}\chi_{i}\cfrac{\partial e_{i}}{\partial t}w\right|
\leq
2D_{i}^{\frac{1}{2}}\left(\int_{\omega_{i}}\kappa|\nabla w|^{2}+\int_{\omega_{i}}\kappa w^{2}\right)^{\frac{1}{2}}\left(\int_{\omega_{i}}\kappa|\nabla e_{i}|^{2}\right)^{\frac{1}{2}}.
\]
Therefore,
\begin{align}
\sum_{i}\left|\int_{\omega_i}\chi_{i}\cfrac{\partial e_{i}}{\partial t}w\right|
& \leq
2\sum_{i}D_{i}^{\frac{1}{2}}\left(\int_{\omega_{i}}\kappa|\nabla w|^{2}+\int_{\omega_{i}}\kappa w^{2}\right)^{\frac{1}{2}}\left(\int_{\omega_{i}}\kappa|\nabla e_{i}|^{2}\right)^{\frac{1}{2}}\nonumber \\
& \leq
2\left(\sum_{i}D_{i}(\int_{\omega_{i}}\kappa|\nabla w|^{2}+\int_{\omega_{i}}\kappa w{}^{2})\right)^{\frac{1}{2}}\left(\sum_{i}\int_{\omega_{i}}\kappa|\nabla e_{i}|^{2}\right)^{\frac{1}{2}}\nonumber \\
& \leq
2D^{\frac{1}{2}}M^{\frac{1}{2}}\left(\int_{\Omega}\kappa|\nabla w|^{2}+\int_{\Omega}\kappa w{}^{2}\right)^{\frac{1}{2}}\left(\sum_{i}\int_{\omega_{i}}\kappa|\nabla e_{i}|^{2}\right)^{\frac{1}{2}},\label{eq:estimate loc H^=00007B-1=00007D norm 2}
\end{align}
where $D=\max\{D_{i}\}$.
Combining (\ref{eq:H^=00007B-1=00007D norm global to local})
with (\ref{eq:estimate loc H^=00007B-1=00007D norm 2}), we have
\[
\int_{T_{n-1}}^{T_{n}}\|\cfrac{\partial(\sum_{i}\chi_{i}e_{i})}{\partial t}\|_{H^{-1}(\kappa,\Omega)}^{2}\leq4DM\sup_{w\in H^{1}(\Omega)}\cfrac{\left(\int_{\Omega}\kappa|\nabla w|^{2}+\int_{\Omega}\kappa w{}^{2}\right)}{\int_{\Omega}\kappa|\nabla w|^{2}}\left(\sum_{i}\int_{T_{n-1}}^{T_{n}}\int_{\omega_{i}}\kappa|\nabla e_{i}|^{2}\right).
\]
Let
\[
E=\sup_{w\in H_{0}^{1}(\Omega)}\cfrac{\int_{\Omega}\kappa|\nabla w|^{2}+\int_{\Omega}\kappa w{}^{2}}{\int_{\Omega}\kappa|\nabla w|^{2}},
\]
then we have
\begin{equation}
\int_{T_{n-1}}^{T_{n}}\|\cfrac{\partial(\sum_{i}\chi_{i}e_{i})}{\partial t}\|_{H^{-1}(\kappa,\Omega)}^{2}\leq4DME\left(\sum_{i}\int_{T_{n-1}}^{T_{n}}\int_{\omega_{i}}\kappa|\nabla e_{i}|^{2}\right).\label{eq:H^=00007B-1=00007D estimate}
\end{equation}

Now, we substitute (\ref{eq:H^=00007B-1=00007D estimate}) and (\ref{eq:V-norm estimate 3}) into (\ref{eq:W-norm estimate}), then we have
\begin{align}
\|\tilde{u}_{h}-P(\tilde{u}_{h})\|_{W^{(n)}}^{2}
& \preceq
M\sum_{i}\left(DE\int_{T_{n-1}}^{T_{n}}\int_{\omega_{i}}\kappa|\nabla e_{i}|^{2}
+\int_{T_{n-1}}^{T_{n}}\int_{\omega_{i}}\kappa|\nabla\chi_i|^{2}e_{i}^{2}
+\int_{\omega_{i}}\chi_{i}^{2}e_{i}^{2}(x,T_{n-1}^{+})\right)\nonumber\\
& \preceq
M\sum_{i}\left(DE\int_{T_{n-1}}^{T_{n}}\int_{\omega_{i}}\kappa|\nabla e_{i}|^{2}
+\int_{T_{n-1}}^{T_{n}}\int_{\omega_{i}}\kappa|\nabla\chi_i|^{2}e_{i}^{2}
+\int_{\omega_{i}}e_{i}^{2}(x,T_{n-1}^{+})\right).\label{eq:W-norm estimate 2}
\end{align}
Note that
\begin{align*}
\int_{T_{n-1}}^{T_{n}}\int_{\omega_{i}}\kappa|\nabla e_{i}|^{2}
& \leq \int_{T_{n-1}}^{T_{n}} \frac{1}{\min_{x\in\omega_i}\{|\chi_i^{+}(x)|^2\}} \int_{\omega_{i}}\kappa|\chi_i^{+}|^2|\nabla e_{i}|^{2}\\
& \leq \frac{1}{\min_{x\in\omega_i}\{|\chi_i^{+}(x)|^2\}} \int_{T_{n-1}}^{T_{n}}\int_{\omega_{i}^{+}}\kappa|\chi_i^{+}|^2|\nabla e_{i}^{+}|^{2}.
\end{align*}
Applying Lemma \ref{lem: Caccioppoli} for $\omega_{i}^{+}$ then implies
\begin{align}
\int_{T_{n-1}}^{T_{n}}\int_{\omega_{i}}\kappa|\nabla e_{i}|^{2}
& \leq \frac{1}{\min_{x\in\omega_i}\{|\chi_i^{+}(x)|^2\}} \left( \int_{T_{n-1}}^{T_{n}}\int_{\omega_{i}^{+}}\kappa|\nabla\chi_i^{+}|^{2}|e_{i}^{+}|^{2}
+ \int_{\omega_i^{+}}|\chi_i^{+}|^{2}|e_i^{+}(x,T_{n-1}^{+})|^{2} \right) \nonumber\\
& \leq F_i \left( \int_{T_{n-1}}^{T_{n}}\int_{\omega_{i}^{+}}\kappa|\nabla\chi_i^{+}|^{2}|e_{i}^{+}|^{2}
+ \int_{\omega_i^{+}}|e_i^{+}(x,T_{n-1}^{+})|^{2} \right), \label{eq:W-norm estimate 3}
\end{align}
where $F_i = \cfrac{1}{\min_{x\in\omega_i}\{|\chi_i^{+}(x)|^2\}}$. Substituting (\ref{eq:W-norm estimate 3}), (\ref{eq:oversample estimate 1}) and (\ref{eq:oversample estimate 2}) into (\ref{eq:W-norm estimate 2}) gives
\begin{align}
\|\tilde{u}_{h}-P(\tilde{u}_{h})\|_{W^{(n)}}^{2}
& \preceq
M\sum_{i}(DEF_i + 1) \left( \int_{T_{n-1}}^{T_{n}}\int_{\omega_{i}^{+}}\kappa|\nabla\chi_i^{+}|^{2}|e_{i}^{+}|^{2}
+ \int_{\omega_i^{+}}|e_i^{+}(x,T_{n-1}^{+})|^{2} \right) \nonumber\\
& \preceq
M(DEF+1)\sum_{i} \left( \int_{T_{n-1}}^{T_{n}}\int_{\omega_{i}^{+}}\widetilde{\kappa}^{+}(x,t)|e_{i}^{+}|^{2}
+ \int_{\omega_i^{+}}|e_i^{+}(x,T_{n-1}^{+})|^{2} \right),\label{eq:W-norm estimate 4}
\end{align}
where $F=\max\{F_{i}\}$.
%Therefore,
%\begin{align*}
%\|\tilde{u}_{h}-P(\tilde{u}_{h})\|_{W^{(n)}}^{2} & \preceq \sum_{i}\left( \int_{T_{n-1}}^{T_{n}}\int_{\omega_{i}^{+}}\kappa|\nabla\chi_i^{+}|^{2}|e_{i}^{+}|^{2}
%+ \int_{\omega_i^{+}}|e_i^{+}(x,T_{n-1}^{+})|^{2} \right).
%\end{align*}
Using the spectral problem, we have
\begin{equation}\label{eq:W-norm estimate 5}
\|\tilde{u}_{h}-P(\tilde{u}_{h})\|_{W^{(n)}}^{2}\preceq
M(DEF+1)\sum_{i}\left(\cfrac{1}{\lambda_{L_{i}+1}^{\omega_i^{+}}}\|\tilde{u}_{h,i}^{+}\|_{V^{(n)}(\omega_i^{+})}^{2}\right).
\end{equation}
Combine (\ref{eq:general estimate}), (\ref{eq:estimate inf}) and (\ref{eq:W-norm estimate 5}), and we finally obtain
\[
\|u_{h}-u_{H}\|_{V^{(n)}}^{2}\preceq
M(DEF+1)\sum_{i}\left(\frac{1}{\lambda_{L_{i}+1}^{\omega_i^{+}}}\|\tilde{u}_{h,i}^{+}\|_{V^{(n)}(\omega_i^{+})}^{2}\right)
+\|u_{h}-\tilde{u}_{h}\|_{W^{(n)}}^{2}+\|u_{h}-u_{H}\|_{V^{(n-1)}}^{2}.
\]

\end{proof}

\section{Numerical results. Offline GMsFEM.} \label{sect:NR1}

In this section, we present a number of representative numerical examples
to show the performance of the proposed method. In particular,
we solve Equation (\ref{eq:para PDE}) using the space-time GMsFEM to
validate the effectiveness of the proposed approaches. The space domain $\Omega$ is taken as the unit square $[0,1]\times[0,1]$ and is divided into $10\times10$ coarse blocks consisting of uniform squares. Each coarse block is then divided into $10\times10$ fine blocks consisting of uniform squares. That is, $\Omega$ is partitioned by $100\times100$ square fine-grid blocks. The whole time interval is $[0, 1.6]$ (i.e., $T = 1.6$) and is divided into two uniform coarse time intervals and each coarse time interval is then divided into $8$ fine time intervals. We also use a source term $f = 1$ and impose a continuous  initial condition $\beta(x,y)=\sin(\pi x)\sin(\pi y)$. We  employ three different high-contrast permeability fields $\kappa(x,t)$'s to examine our method, which will be shown in the following three cases separately. In each case, we first solve for $u_h$ from Equation (\ref{eq:fine problem}) to obtain  the fine-grid solution. Then we solve for the multiscale solution $u_H$ using the space-time GMsFEM. To compare the accuracy, we will use the following error quantities:
\begin{equation}\label{err_formulus}
 e_1 =\left( \frac{\int_0^{T} \|u_{H}(t)-u_{h}(t)\|^2_{L^2(\Omega)}}{\int_0^{T}\|u_{h}(t)\|^2_{L^2(\Omega)}} \right)^{1/2}, \qquad
 e_2 =\left( \frac{\int_0^{T} \int_{\Omega} \kappa |\nabla(u_{H}(t)-u_{h}(t))|^2}{\int_0^{T} \int_{\Omega} \kappa |\nabla u_{h}(t)|^2} \right)^{1/2}.
\end{equation}

Since we are using the technique of randomized oversampling in the computation of the snapshot space, we would like to introduce the concept of $\emph{snapshot ratio}$, which is calculated as the number of randomized snapshots divided by the number of the full snapshots on one coarse neighborhood $\omega_i$. Here, the number of the full snapshots refers to the number of functions $\delta_i(x,t)$ from Equation (\ref{eq:locSnap}). In the following experiment with $100\times100$ fine-grid mesh, this number of the full snapshots on each coarse neighborhood is calculated by $n_{\text{total}}^{\text{snap}}=21\times21+40\times8 =761$.

\subsection{High-contrast Permeability Field 1: High-contrast medium translated in time}\label{sect:media1}

We start with a high-contrast permeability field $\kappa(x,t)$, which is translated uniformly after every other fine time step. High-contrast permeability fields at the initial and final time steps are shown in Figure \ref{Figure:UniCoeff}. Next, we consider applying the space-time GMsFEM to  Equation (\ref{eq:para PDE}) and solve for the multiscale solution $u_H$. Recall the procedures that are described in the Section \ref{sect:GMsFEM}, where we need to construct the snapshot spaces in the first place. The number of local offline basis that will be used in each $\omega_i$, denoted by $L_i$, and the buffer number $p_{\text{bf}}$ needs to be chosen in advance since they determine how many local snapshots are used. Then we can construct the lower dimensional offline space by performing space reduction on the snapshot space. In our experiments, we use the same  buffer number and the same number of local offline basis for all coarse neighborhood $\omega_i$'s.

\begin{figure}[tbp]
\centering
\begin{minipage}[t]{0.4\linewidth}
\centering
\includegraphics[width=\columnwidth]{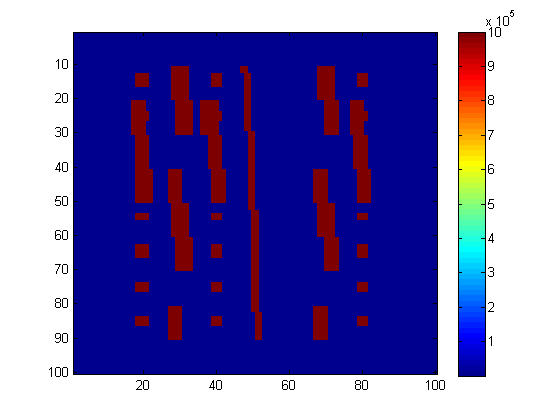}
\end{minipage}
\begin{minipage}[t]{0.4\linewidth}
\centering
\includegraphics[width=\columnwidth]{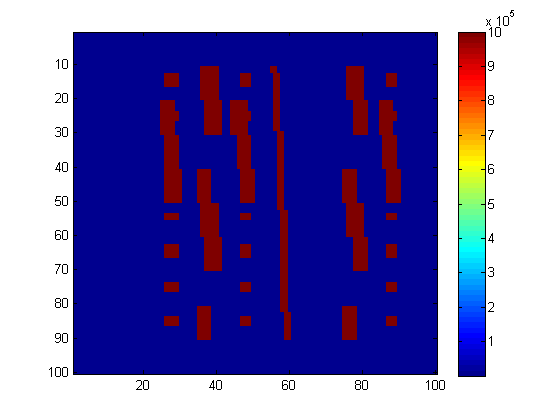}
\end{minipage}
\caption{High-contrast Permeability Field 1. Left: the permeability at
the initial time. Right: the permeability at the final time.}
\label{Figure:UniCoeff}
\end{figure}

First, we fix $L_i = 11$ for all $\omega_i$'s and examine the influences of various buffer numbers on the solution errors $e_1$ and $e_2$. The results are displayed in the left table of Table \ref{Table:test bf Li 1}. It is observed that when increasing the buffer numbers, one can get more accurate solutions, which is as expected. But the error decays very slowly, which indicates that using different buffer numbers doesn't affect the convergence rate too much. Based on this observation, it is not necessary to choose a large buffer number in order to improve convergence rate. Then we consider the choice of $L_i$, the number of eigenbasis in a neighborhood. With the fixed buffer number $p_{\text{bf}}=8$, we examine the convergence behaviors of using different $L_i$'s. Relative errors of multiscale solutions are shown in the right table of Table \ref{Table:test bf Li 1}. We observe that with a fixed buffer number, the relative errors are decreasing as using more offline basis. To see a more quantitative relationship between the relative errors and the values of $L_i$ as well as being inspired by the result in Theorem \ref{main thm}, we inspect the values of $1/\Lambda_{*}$ and the corresponding squared errors (see Table \ref{Table:minlambda} and Figure \ref{Figure:LambdaVSError}), where $\Lambda_{*}=\min_{\omega_i} \lambda_{L_i + 1}^{\omega_i}$ and $\{\lambda_j^{\omega_i}\}$ are the eigenvalues associated with the eigenbasis computed by spectral problem (\ref{eq:eig-problem}) in each $\omega_i$.
We note that when plotting Figure \ref{Figure:LambdaVSError}, we don't use the values of case $L_i=2$, because in this case as in the case with one basis function
per node, the method does not converge as we do not have sufficient
number of basis functions.
%One reason is that $e_2$ is so large (over $100\%$) that we believe this case is not reliable. Another reason is that the relative large values of $1/\Lambda_{*}$ and $e_2$ in this case would obscure the trend in the other cases if we draw these two point out in the figures.
We note that the two curves in Figure \ref{Figure:LambdaVSError} track each other somewhat closely. This indicates that $1/\Lambda_{*}$'s and $e_2^2$'s are
 correlated and we calculate for the correlation coefficient to be
$corrcoef(1/\Lambda_{*}, e_2^2)= 0.9778.$
%One can see the cross-correlation coefficient $0.9778$ is very close to $1$.
%This implies a linear relationship between the squared relative error $e_2^2$ and $1/\Lambda_{*}$, which is consistent with our result in Theorem \ref{main thm}.

Observing the dimensions of the offline spaces $V_{\text{off}}$, one can see that compared with the traditional fine-scale finite element method, the proposed space-time GMsFEM uses much fewer degrees of freedom while achieving an accurate solution. Also, by inspecting the snapshot ratios, one can see that the use of randomization can reduce the dimension of snapshot spaces substantially. We would like to comment that oversampling technique is necessary for the randomization. For example, in the case $L_i=6$ and $p_{\text{bf}}=8$, if without oversampling the errors $e_1$ and $e_2$ are $11.19\%$ and $88.42\%$, respectively, which are worse than the errors obtained with oversampling.

\begin{table}[H]
\centering
 \begin{minipage}[t]{0.4\linewidth}
 \centering
  \begin{tabular}{ | c | c | c | c |  }
    \hline
    $p_{\text{bf}}$  &Snapshot ratio  &$e_1$  &$e_2$  \\
    \hline
    1   & 0.0158     & 6.18\%    & 53.90\%   \\
    4   & 0.0197     & 5.66\%    & 48.04\%   \\
    8   & 0.0250     & 5.17\%    & 45.86\%   \\
    12  & 0.0302     & 5.16\%    & 43.83\%   \\
    20  & 0.0407     & 4.71\%    & 41.14\%   \\
    30  & 0.0539     & 4.35\%    & 38.68\%   \\
    40  & 0.0670     & 4.23\%    & 37.60\%   \\
    \hline
  \end{tabular}
  \end{minipage}
  \begin{minipage}[t]{0.5\linewidth}
  %\centering
  \begin{tabular}{ | c | c | c | c | c |  }
    \hline
    $L_i$  &$dim(V_{\text{off}})$   &Snapshot ratio  &$e_1$  &$e_2$  \\
    \hline
    2     &162   & 0.0131   &17.03\%    & 129.14\%   \\
    6     &486   & 0.0184   &8.11\%    & 62.59\%  \\
    10    &810   & 0.0237   &6.97\%    & 54.85\%   \\
    20    &1620  & 0.0368   &4.81\%    & 41.18\%   \\
    30    &2430  & 0.0499   &3.29\%    & 31.64\%   \\
    40    &3240  & 0.0631   &2.28\%    & 24.43\%    \\
    50    &4050  & 0.0762   &1.54\%    & 18.45\%   \\
    \hline
  \end{tabular}
  \end{minipage}
\caption{First permeability field. Left: errors with the fixed number of offline basis $L_i=11$. Right: errors with the fixed buffer number $p_{\text{bf}}=8$. }
\label{Table:test bf Li 1}
\end{table}

\begin{table}[H]
  \centering
  \begin{tabular}{ | c | c |c|c|  }
    \hline
    $L_i$    &$1/\Lambda_{*}$  &$e_1^2$  &$e_2^2$ \\
    \hline \hline
    2      &0.2734    &2.90\%     &166.78\% \\
    6      &0.0120    &0.66\%     &39.17\%  \\
    10      &0.0085   &0.49\%     &30.08\%  \\
    20      &0.0061   &0.23\%     &16.96\%  \\
    30      &0.0053   &0.11\%     &10.01\% \\
    40      &0.0048   &0.05\%     & 5.97\%  \\
    50      &0.0042   &0.02\%     & 3.40\% \\
    \hline
  \end{tabular}
\caption{ $1/\Lambda_{*}$ values and errors.}
\label{Table:minlambda}
\end{table}

\begin{figure}[H]
\centering
\begin{minipage}[t]{0.4\linewidth}
\centering
\includegraphics[width=\columnwidth]{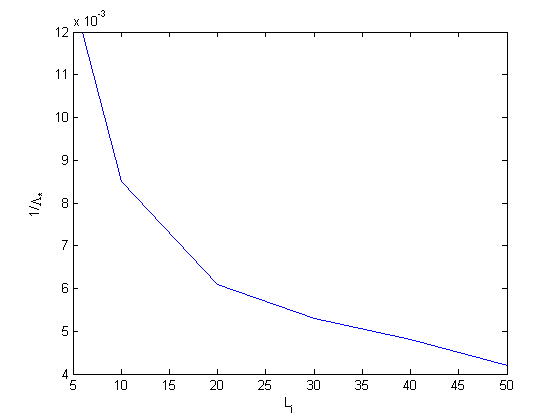}
\end{minipage}
\begin{minipage}[t]{0.4\linewidth}
\centering
\includegraphics[width=\columnwidth]{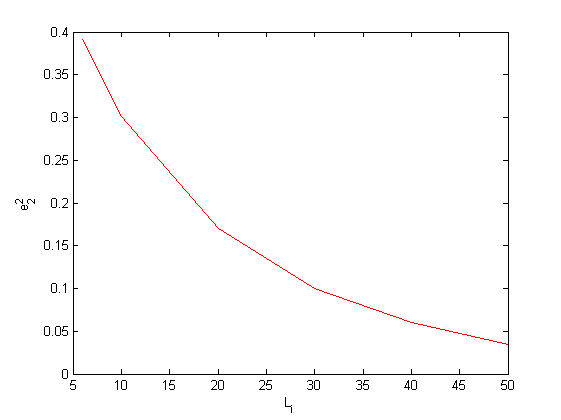}
\end{minipage}
\caption{Left: $1/\Lambda_{*}$ vs $L_i$; Right: $e_2^2$ vs $L_i$.}
\label{Figure:LambdaVSError}
\end{figure}

\subsection{High-contrast Permeability Field 2: Four channels translated in time}

In this subsection, we consider a more structured high-contrast permeability field $\kappa(x,t)$, which has four channels inside and these four channels are translated uniformly in time. High-contrast permeability fields at the initial and final time steps are shown in Figure \ref{Figure:4chanlsTran}. We repeat our steps from the previous example by fixing $L_i$ and $p_{\text{bf}}$, separately. The results are shown in Table \ref{Table:test bf Li 2}. One can still observe that increasing the buffer numbers will slowly reduce the relative errors and with a fixed buffer number, the relative errors are decreasing as adding more offline basis. Using a similar approach, we can also get the cross-correlation coefficient between $e_2^2$ and $1/\Lambda_{*}$, which is $0.9863$. This  suggests a linear relationship between $e_2^2$ and $1/\Lambda_{*}$ and verifies Theorem \ref{main thm}.

\begin{figure}[H]
\centering
\begin{minipage}[t]{0.4\linewidth}
\centering
\includegraphics[width=\columnwidth]{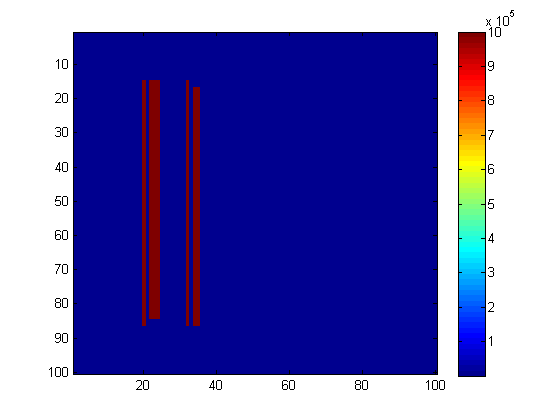}
\end{minipage}
\begin{minipage}[t]{0.4\linewidth}
\centering
\includegraphics[width=\columnwidth]{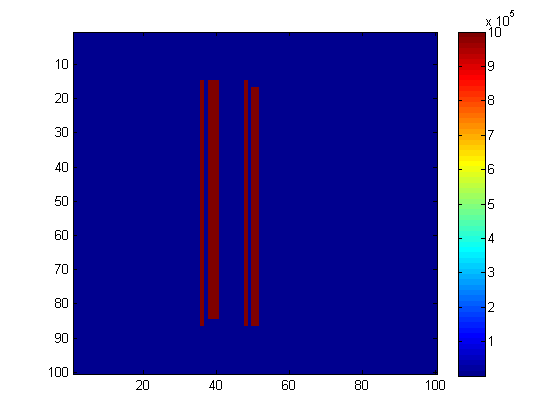}
\end{minipage}
\caption{High-contrast Permeability Field 2. Left: the permeability at the initial time. Right: the permeability at the final time.}
\label{Figure:4chanlsTran}
\end{figure}

\begin{table}[H]

 \begin{minipage}[t]{0.4\linewidth}
 \centering
  \begin{tabular}{ | c | c | c | c |  }
    \hline
    $p_{\text{bf}}$  &Snapshot ratio  &$e_1$  &$e_2$  \\
    \hline
    1   & 0.0158     & 7.42\%    & 61.87\%   \\
    4   & 0.0197     & 7.30\%    & 58.95\%   \\
    8   & 0.0250     & 7.14\%    & 57.30\%   \\
    12  & 0.0302     & 7.00\%    & 54.01\%   \\
    20  & 0.0407     & 6.81\%    & 50.85\%   \\
    30  & 0.0539     & 6.61\%    & 49.30\%   \\
    40  & 0.0670     & 6.43\%    & 48.26\%   \\
    \hline
  \end{tabular}
  \end{minipage}
  \begin{minipage}[t]{0.5\linewidth}
  %\centering
  \begin{tabular}{ | c | c | c | c | c |  }
    \hline
    $L_i$  &$dim(V_{\text{off}})$   &Snapshot ratio  &$e_1$  &$e_2$  \\
    \hline
    2     &162   & 0.0131   &11.91\%    & 104.95\%   \\
    6     &486   & 0.0184   &8.33\%    & 70.82\%  \\
    10    &810   & 0.0237   &7.25\%    & 58.25\%   \\
    20    &1620  & 0.0368   &5.67\%    & 43.10\%   \\
    30    &2430  & 0.0499   &3.90\%    & 32.75\%   \\
    40    &3240  & 0.0631   &2.73\%    & 27.08\%    \\
    50    &4050  & 0.0762   &1.86\%    & 20.70\%   \\
    \hline
  \end{tabular}
  \end{minipage}
\caption{Second permeability field. Left: errors with the fixed number of offline basis $L_i=11$. Right: errors with the fixed buffer number $p_{\text{bf}}=8$. }
\label{Table:test bf Li 2}
\end{table}

\subsection{High-contrast Permeability Field 3: Four channels rotated in time}

In the third example, we consider another structured high-contrast permeability field $\kappa(x,t)$ which has four channels inside and these four channels are rotated anticlockwise around the center by $11.25$ degrees after each fine time step. High contrast permeability fields at the initial time step is shown in Figure \ref{Figure:4chanlsRota}. We repeat the same procedures as in the previous two examples. The results are shown in Table \ref{Table:test bf Li 3} and one can draw similar conclusions as before. The cross-correlation coefficient between $e_2^2$ and $1/\Lambda_{*}$ is calculated as $0.9959$. This shows a linear relationship between $e_2^2$ and $1/\Lambda_{*}$ (see Theorem \ref{main thm}).

\begin{figure}[H]
\centering
\includegraphics[width=7cm]{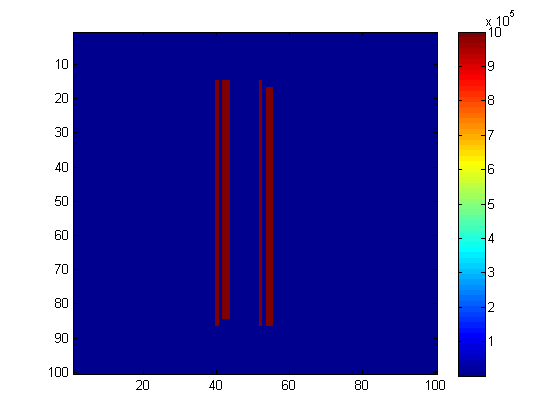}
\caption{High-contrast Permeability Field 3 at the initial time.}
\label{Figure:4chanlsRota}
\end{figure}

\begin{table}[H]

 \begin{minipage}[t]{0.4\linewidth}
 \centering
  \begin{tabular}{ | c | c | c | c |  }
    \hline
    $p_{\text{bf}}$  &Snapshot ratio  &$e_1$  &$e_2$  \\
    \hline
    1   & 0.0158     & 8.68\%    & 72.86\%   \\
    4   & 0.0197     & 8.67\%    & 71.67\%   \\
    8   & 0.0250     & 8.56\%    & 71.42\%   \\
    12  & 0.0302     & 8.44\%    & 68.87\%   \\
    20  & 0.0407     & 8.18\%    & 65.88\%   \\
    30  & 0.0539     & 7.96\%    & 61.56\%   \\
    40  & 0.0670     & 7.58\%    & 57.58\%   \\
    \hline
  \end{tabular}
  \end{minipage}
  \begin{minipage}[t]{0.5\linewidth}
  %\centering
  \begin{tabular}{ | c | c | c | c | c |  }
    \hline
    $L_i$  &$dim(V_{\text{off}})$   &Snapshot ratio  &$e_1$  &$e_2$  \\
    \hline
    2     &162   & 0.0131   &10.41\%    & 109.40\%   \\
    6     &486   & 0.0184   &9.40\%    & 83.60\%  \\
    10    &810   & 0.0237   &8.63\%    & 70.84\%   \\
    20    &1620  & 0.0368   &7.42\%    & 57.66\%   \\
    30    &2430  & 0.0499   &6.14\%    & 47.78\%   \\
    40    &3240  & 0.0631   &4.75\%    & 39.89\%    \\
    50    &4050  & 0.0762   &3.29\%    & 30.11\%   \\
    \hline
  \end{tabular}
  \end{minipage}
\caption{Third permeability field. Left: errors with the fixed number of offline basis $L_i=11$. Right: errors with the fixed buffer number $p_{\text{bf}}=8$. }
\label{Table:test bf Li 3}
\end{table}

\section{Residual based online adaptive procedure}\label{sect:online}

As we observe in the previous examples, the offline errors do not decrease rapidly after several multiscale functions are selected. In these cases, online basis functions can help to reduce the error and obtain an accurate approximation of the fine-scale solution \cite{Chung:2015:ROG:2837849.2838155}. The use of online basis functions gives a rapid convergence. Next, we  will derive a framework for the construction of online multiscale basis functions.

We use the index $m\geq1$ to represent the online enrichment level. At the enrichment level $m$, we use $V_{ms}^m$ to denote the corresponding space-time GMsFEM space and $u_{ms}^m$ the  corresponding solution obtained in (\ref{eq:space-time FEM coarse decoupled}). The sequence of functions $\{u_{ms}^m\}_{m\geq1}$ will converge to the fine-scale solution. We emphasize that the space $V_{ms}^m$ can contain both offline and online basis functions, and define $V_{ms}^0 = V_{\text{off}}$. We will construct a strategy for getting the space $V_{ms}^{m+1}$ from $V_{ms}^m$.

Next we present a framework for the construction of online basis functions. By online basis
functions, we mean basis functions that are computed during the iterative process using the residual. This is the contrary to offline basis functions that are computed before the iterative process. The online basis functions for enrichment level $m+1$ are computed based on some local residuals for the multiscale solution $u_{ms}^m$. Thus, we see that some offline basis functions are necessary for the computations of online basis functions. In our numerical examples from the following section, we will also see how many offline basis functions are needed in order to obtain a rapidly converging sequence of solutions.

For brevity, we denote the left hand side of (\ref{eq:space-time FEM coarse decoupled}) by $a(u_{ms}^{(n)},v)$ and the right hand side $F(v)$. That is, the solution $u_{ms} = \oplus_{n=1}^{N}u_{ms}^{(n)}$ where $u_{ms}^{(n)}$ satisfies
\[
a(u_{ms}^{(n)},v) = F(v), \quad \forall v\in V_{H}^{(n)}.
\]
Consider a given coarse neighborhood $\omega_i$. Suppose that at the enrichment level $m$, we need to add an online basis function $\phi\in V_h$ in $\omega_i$. Then the required $\phi= \oplus_{n=1}^{N}\phi^{(n)}$ satisfies that $\phi^{(n)}$ is the solution of
\[
a(\phi^{(n)},v) = R^{(n)}(v), \quad \forall v\in V_{h},
\]
where $R^{(n)}(v)= F(v) - a(u_{ms}^{m(n)},v)$ is the online residual at the coarse time interval $[T_{n-1},T_{n}]$.

In the following, we would like to form a residual based online algorithm in each coarse time interval $[T_{n-1},T_{n}]$, see Algorithm \ref{online_algorithm}. For simplicity, we will omit the time index $(n)$ on the spaces and solutions in this description. We consider enrichment on non-overlapping coarse neighborhoods. Thus, we divide the $\{\omega_i\}_{i=1}^{N_c}$ into $P$ non-overlapping groups and denote each group by $\{\omega_i\}_{i\in I_p}$, $p=1,...,P$. We denote by $M$ the number of online iterations.

\begin{algorithm}[htb] \caption{Residual based online algorithm} \label{online_algorithm}
\begin{algorithmic}[1]

\State \textbf{Initialization:} Offline space $V_{ms}^0 = V_{\text{off}}$, offline solution $u_{ms}^0 = u_{\text{H}}$.

\medskip

\For{$m=0$ to $M$:}
    \For{$p=1$ to $P$}
        \State (1) On each $\omega_i (i\in I_p)$, compute residual $R^{m}(v) = a(u_{ms}^{m},v) - F(v),\quad v\in V_{h}$.

        \State (2) For each $i$, solve $a(\phi_i,v) = R^{m}(v), \quad \forall v\in V_{h}$.

        \State (3) Set $V_{ms}^m$ = $V_{ms}^m\cup\{\phi_i | i\in I_p\}$. \label{online_adding}

        \State (4) Solve for a new $u_{ms}^m \in V_{ms}^m$ satisfying $a(u_{ms}^m,v) = F(v), \quad \forall v\in V_{ms}^m$.
    \EndFor
    \State Set $V_{ms}^{m+1}$ = $V_{ms}^m$, and $u_{ms}^{m+1}$ = $u_{ms}^m$.
\EndFor

\end{algorithmic}
\end{algorithm}

To further improve the convergence and efficiency of the online method, we can adopt an online adaptive procedure. In this adaptive approach, the online enrichment is performed for coarse neighborhoods that have a cumulative residual that is $\theta$ fraction of the total residual. More precisely, assume that the $V^{(n)}$ norm of local residuals on $\{\omega_i|i\in I_p\}$, denoted by $\{r_i|i\in I_p\}$, are arranged so that
\[
r_{p_1} \geq r_{p_2} \geq r_{p_3} \geq \cdot\cdot\cdot \geq r_{p_J},
\]
where we suppose $I_p = \{p_1, p_2, p_3,\cdot\cdot\cdot, p_J\}$. Instead of adding $\{\phi_i | i\in I_p\}$ into $V_{ms}^m$ at step \ref{online_adding} in Algorithm \ref{online_algorithm}, we only add the basis $\{\phi_1,\cdot\cdot\cdot,\phi_k\}$ for the corresponding coarse neighborhoods such that $k$ is the smallest integer satisfying
\[
\Sigma_{i=1}^{k}r_{p_i}^2 \geq \theta\Sigma_{i=1}^{J}r_{p_i}^2.
\]
In the examples below, we will see that the proposed adaptive procedure gives a better convergence and is more efficient.

\section{Numerical results. Online GMsFEM}\label{sect:NR2}

In this section, we present numerical examples to demonstrate the performance of the proposed online method in solving Equation (\ref{eq:para PDE}). To implement the space-time online GMsFEM, we will first choose a fixed number of offline basis functions for every coarse neighborhood, and calculate the resulting offline space $V_{\text{off}}$. Then we conduct the online process by following Algorithm \ref{online_algorithm}. In this experiment, we use the same space-time domain and mesh (coarse and fine), the same source term $f$ and initial condition $\beta(x,y)$, the same definitions of relative errors $e_1$ and $e_2$, as in Section \ref{sect:NR1}. The permeability field $\kappa(x,t)$ is chosen as the high-contrast permeability field 1 from Section \ref{sect:media1}. The buffer number in the computation of snapshot space is chosen to be 8.

First, we implement the space-time online GMsFEM by choosing different numbers of offline basis functions ($L_i = 1,2,3,4,5$) on every coarse neighborhood. The relative errors of online solutions are presented in Table \ref{L2ErrTable_online} and Table \ref{H1ErrTable_online}. Note that in the first column, we show the number of basis functions used for each coarse neighborhood $\omega_i$, and the degrees of freedom (DOF) of multiscale space on each coarse time interval which are the numbers in parentheses, after online enrichment. For example, $2(162)$ in the first column means that after online enrichment, $2$ multiscale basis are used on each $\omega_i$ and the DOF of multiscale space on each coarse time interval is $162$. And if we initially choose $L_i=1$, then it means $1$ online iteration is performed, which add $1$ online basis to each $\omega_i$. If $L_i=2$ initially, then it means we do not perform any online iteration and  $2$ multiscale basis are offline basis functions. By observing each column, one can see that the errors decay fast with more online iterations being performed. This is observed for both $e_1$ and $e_2$ when $L_i\geq4$. This suggests that in this specific setting, we can get a fast online convergence with $4$ offline basis chosen on each $\omega_i$. After a small number of online iterations, the relative errors decrease to  a significantly small level. We consider reducing the high contrast of the permeability field $\kappa(x,t)$ from $10^6$ to $100$. Then we look at the relative errors of online multiscale solutions (see Table \ref{L2ErrTable_online2} and Table \ref{H1ErrTable_online2}). The same phenomena can be observed except that the fast online convergence rate can be achieved for any choice of $L_i$. This implies that the number of offline basis functions used to guarantee a fast online convergence rate is  related to the high contrast of the permeability field.

\begin{table}[H]
  \centering
  \begin{tabular}{ | c | c | c | c | c | c |}
    \hline
    $DOF$  &$e_1$($L_i=1$)  &$e_1$($L_i=2$)  &$e_1$($L_i=3$)  &$e_1$($L_i=4$) &$e_1$($L_i=5$)\\
    \hline \hline
    1(81)   & 97.57\%     &-            & -           & -            & -                  \\
    \hline
    2(162)   & 93.20\%      &96.71\%      & -           & -            & -                 \\
    \hline
    3(243)   & 44.24\%      &23.22\%       & 21.27\%      & -            & -                \\
    \hline
    4(324)   & 15.37\%      &6.53\%    & 7.17e-1\%   & 10.20\%       & -                \\
    \hline
    5(405)   & 8.65\%      &3.69\%    & 2.06e-1\%   & 2.58e-1\%     &5.20\%            \\
    \hline
    6(486)   & 5.15\%    &1.71\%     & 5.41e-2\%    & 1.75e-2\%      &1.06e-1\%           \\
    \hline
    7(567)   & 2.58\%    &3.11e-1\%    & 5.54e-3\%   & 6.12e-4\%     &2.99e-3\%           \\
    \hline
  \end{tabular}
\caption{Relative online errors $e_1$, with the different numbers of offline basis functions. High contrast = $10^6$.}
\label{L2ErrTable_online}
\end{table}

\begin{table}[H]
  \centering
  \begin{tabular}{ | c | c | c | c | c | c |}
    \hline
    $DOF$  &$e_2$($L_i=1$)  &$e_2$($L_i=2$)  &$e_2$($L_i=3$)  &$e_2$($L_i=4$) &$e_2$($L_i=5$)\\
    \hline \hline
    1(81)   & 138\%        &-            & -           & -            & -                  \\
    \hline
    2(162)   & 113\%      &114\%       & -             & -            & -                 \\
    \hline
    3(243)   & 84.93\%    &139\%      & 104\%         & -            & -                \\
    \hline
    4(324)   & 82.48\%   &82.08\%       & 11.43\%        & 73.50\%       & -                \\
    \hline
    5(405)   & 69.15\%   &51.13\%       & 3.29\%      & 4.78\%        &48.26\%            \\
    \hline
    6(486)   & 51.17\%   &34.00\%       & 1.01\%      & 3.53e-1\%     &1.86\%           \\
    \hline
    7(567)   & 37.93\%    &7.81\%    & 1.05e-1\%      & 9.89e-3\%     &4.75e-2\%           \\
    \hline
  \end{tabular}
\caption{Relative online errors $e_2$, with the different numbers of offline basis functions. High contrast = $10^6$.}
\label{H1ErrTable_online}
\end{table}

\begin{table}[H]
  \centering
  \begin{tabular}{ | c | c | c | c | c | c |}
    \hline
    $DOF$  &$e_1$(1 basis)  &$e_1$(2 basis)  &$e_1$(3 basis)  &$e_1$(4 basis) &$e_1$(5 basis)\\
    \hline \hline
    1(81)   & 19.28\%     &-            & -           & -            & -                  \\
    \hline
    2(162)   & 1.97\%      &13.03\%      & -           & -            & -                 \\
    \hline
    3(243)   & 2.81e-1\%   &9.81e-1\%    & 9.27\%      & -            & -                \\
    \hline
    4(324)   & 3.48e-2\%   &1.24e-1\%    & 2.23e-1\%   & 8.34\%       & -                \\
    \hline
    5(405)   & 1.89e-3\%   &1.11e-2\%    & 9.70e-2\%   & 2.09e-1\%     &7.38\%            \\
    \hline
    6(486)   & 2.67e-5\%   &1.33e-4\%    & 2.07e-4\%   & 8.71e-3\%     &1.56e-1\%           \\
    \hline
    7(567)   & 2.51e-7\%   &9.32e-7\%    & 1.45e-6\%   & 1.16e-4\%     &8.62e-3\%           \\
    \hline
  \end{tabular}
\caption{Relative online errors $e_1$, with the different numbers of offline basis functions. High contrast = $100$.}
\label{L2ErrTable_online2}
\end{table}

\begin{table}[H]
  \centering
  \begin{tabular}{ | c | c | c | c | c | c |}
    \hline
    $DOF$  &$e_2$(1 basis)  &$e_2$(2 basis)  &$e_2$(3 basis)  &$e_2$(4 basis) &$e_2$(5 basis)\\
    \hline \hline
    1(81)   & 219\%        &-            & -           & -            & -                  \\
    \hline
    2(162)   & 14.75\%      &123\%       & -             & -            & -                 \\
    \hline
    3(243)   & 3.35\%       &8.37\%      & 81.80\%         & -            & -                \\
    \hline
    4(324)   & 4.03e-1\%   &1.11\%       & 2.63\%         & 67.86\%       & -                \\
    \hline
    5(405)   & 2.11e-2\%   &1.01e-1\%    & 1.68e-1\%      & 2.29\%        &59.93\%            \\
    \hline
    6(486)   & 5.61e-4\%   &1.64e-3\%    & 3.71e-3\%      & 1.35e-1\%     &1.77\%           \\
    \hline
    7(567)   & 4.57e-6\%   &1.72e-5\%    & 2.29e-5\%      & 2.08e-3\%     &1.41e-1\%           \\
    \hline
  \end{tabular}
\caption{Relative online errors $e_2$ with the different numbers of offline basis functions. High contrast = $100$.}
\label{H1ErrTable_online2}
\end{table}

Next, we perform online adaptive basis construction procedure with $\theta = 0.7$. The numerical results for using $3$, $4$, and $5$ offline basis per coarse neighborhood are shown in Table \ref{Table:AdapUniMedia}. Notice that "$M1+M2$" in the DOF columns means $M1$ degrees of freedom are used on the first coarse time interval and $M2$ degrees of freedom on the second coarse time interval. To compare the behaviors of online processes with and without adaptivity, we plot out the log values of $e_2$ against DOFs. See Figure \ref{Fig:AdapUniMedia}. We observe that to achieve a certain error, fewer online basis functions are needed with adaptivity. This indicates that the proposed adaptive procedure gives us better convergence and is more efficient.

\begin{table}[H]
  \centering
  \begin{tabular}{ || c | c || c | c || c | c ||}
    \hline
    \multicolumn{2}{||c||}{3 offline basis}  &\multicolumn{2}{c||}{4 offline basis}  &\multicolumn{2}{c||}{5 offline basis}\\
    \hline
    DOF     & $e_2$    & DOF     & $e_2$    & DOF     & $e_2$\\
    \hline
    243+243     & 104\%     &324+324      &73.50\%      & 405+405   &48.26\%  \\
    \hline
    323+322     & 10.57\%   &399+401      &3.56\%      & 471+473   &1.95\%\\
    \hline
    403+392     & 1.49\%    &468+466     &2.13e-1\%    &533+536   &1.21e-1\% \\
    \hline
    480+465     & 9.81e-2\%   &541+529    &1.03e-2\%    &599+603   &6.81e-3\%\\
    \hline
    552+533     & 4.24e-3\%   &611+601    &5.00e-4\%    &670+669   &3.41e-4\%\\
    \hline
  \end{tabular}
\caption{Relative online adaptive errors $e_2$ with different numbers of offline basis functions.}
\label{Table:AdapUniMedia}
\end{table}

\begin{figure}[H]
\begin{minipage}[t]{0.32\linewidth}
\centering
\includegraphics[width=\columnwidth]{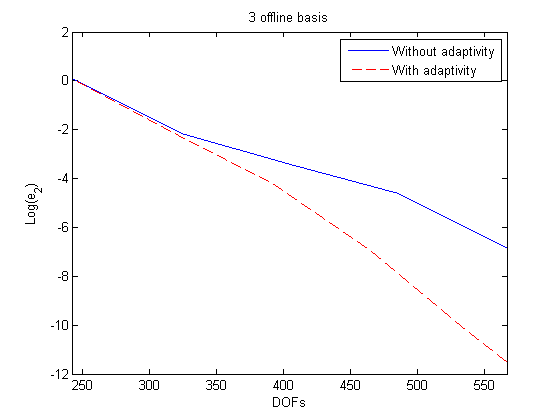}
\end{minipage}
\begin{minipage}[t]{0.32\linewidth}
\centering
\includegraphics[width=\columnwidth]{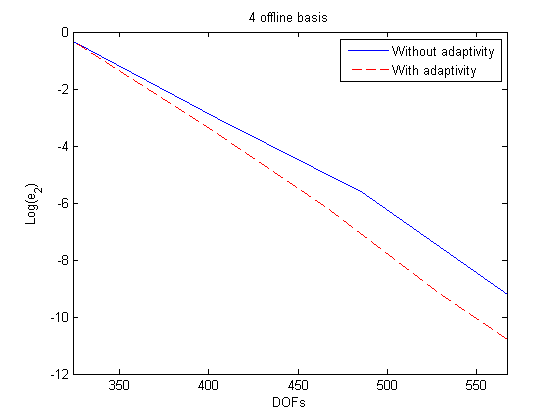}
\end{minipage}
\begin{minipage}[t]{0.32\linewidth}
\centering
\includegraphics[width=\columnwidth]{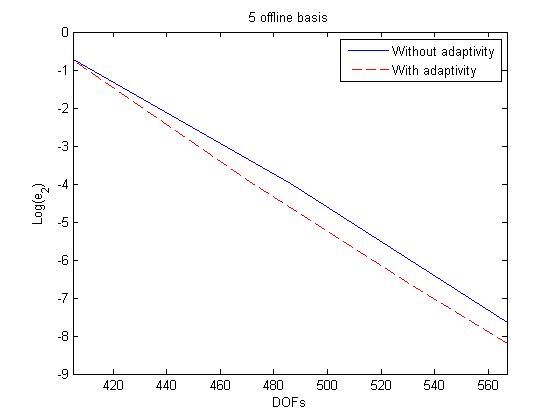}
\end{minipage}
\caption{Adaptivity v.s. no adaptivity.}
\label{Fig:AdapUniMedia}
\end{figure}

\section{Conclusion}\label{sect:conclusion}

In this paper, we consider the construction of the space-time GMsFEM
to solve space-time heterogeneous parabolic equations.
The main ingredients of our approach are (1) the construction of space-time
snapshot vectors, (2) the local spectral decomposition in the snapshot space.
To construct the snapshot vectors, we solve local problems in local space-time
domains. A complete snapshot space will consist of the use of all possible boundary
and initial conditions. However, this can result to very large computational
cost and a high dimensional snapshot space. For this reason, we compute a number of randomized snapshot vectors. In fact, the
number of snapshot vectors is slightly larger than that of the multiscale basis functions used in the simulations. To perform local spectral decomposition, we discuss a couple of choices for local eigenvalue problems motivated by the analysis. We present a convergence analysis of the proposed method. Several numerical examples are presented. In particular, we consider examples where the space-time permeability fields have high contrast and these high-conductivity regions move in the space. If only spatial multiscale basis functions are used, it will require a large dimensional space. Thanks to the space-time multiscale space, we can approximate the problem with a fewer degrees of freedom. Our numerical results show that one can obtain accurate solutions. We also discuss online procedures, where new multiscale basis functions are constructed using the residual. These basis functions are computed in each local space-time domain. Using online basis functions adaptively, one can reduce the error substantially at a cost of online computations.

In this paper, our main objective is to develop systematic
multiscale model reduction techniques
in space-time cells by constructing local (in space-time)
multiscale basis functions. The proposed concepts can be used
for other applications, where one needs space-time multiscale basis
functions.

\section{Appendix}
\subsection{Proof of Lemma \ref{lem: cea lemma}}

\begin{proof}
By the definition of $\|\cdot\|_{V^{(n)}},$
\begin{eqnarray}
\|u_{h}-u_{H}\|_{V^{(n)}}^{2} & = & \cfrac{1}{2}\int_{\Omega}(u_{h}-u_{H})^{2}|_{t=T_{n}^{-}}+\cfrac{1}{2}\int_{\Omega}(u_{h}-u_{H})^{2}|_{t=T_{n-1}^{+}}+\int_{T_{n-1}}^{T_{n}}\int_{\Omega}\kappa|\nabla(u_{h}-u_{H})|^{2}\nonumber
\\
& = & \cfrac{1}{2}\int_{T_{n-1}}^{T_{n}}\int_{\Omega}\cfrac{\partial}{\partial t}(u_{h}-u_{H})^{2}+\int_{\Omega}(u_{h}-u_{H})^{2}|_{t=T_{n-1}^{+}}+\int_{T_{n-1}}^{T_{n}}\int_{\Omega}\kappa|\nabla(u_{h}-u_{H})|^{2}\nonumber \\
 & = & \int_{T_{n-1}}^{T_{n}}\int_{\Omega}\cfrac{\partial(u_{h}-u_{H})}{\partial t}(u_{h}-u_{H})+\int_{\Omega}(u_{h}-u_{H})^{2}|_{t=T_{n-1}^{+}} \nonumber \\
 &  & + \int_{T_{n-1}}^{T_{n}}\int_{\Omega}\kappa|\nabla(u_{h}-u_{H})|^{2}\nonumber\\
 & = & \int_{T_{n-1}}^{T_{n}}\int_{\Omega}\cfrac{\partial(u_{h}-u_{H})}{\partial t}(u_{h}-w)
       + \int_{\Omega}(u_{h}-u_{H})(u_{h}-w)|_{t=T_{n-1}^{+}}\nonumber \\
 &  &  + \int_{T_{n-1}}^{T_{n}}\int_{\Omega}\kappa\nabla(u_{h}-u_{H})\cdot\nabla(u_{h}-w)
       + \int_{T_{n-1}}^{T_{n}}\int_{\Omega}\cfrac{\partial(u_{h}-u_{H})}{\partial t}(w-u_{H})\nonumber \\
 &  & + \int_{\Omega}(u_{h}-u_{H})(w-u_{H})|_{t=T_{n-1}^{+}}
      + \int_{T_{n-1}}^{T_{n}}\int_{\Omega}\kappa\nabla(u_{h}-u_{H})\cdot\nabla(w-u_{H})\label{eq:V norm equality}.
\end{eqnarray}
%Combining (\ref{eq:space-time FEM coarse decoupled}) and (\ref{eq:space-time FEM fine decoupled}),
From (\ref{eq:space-time FEM coarse decoupled}) and the similar formulation for fine scale solution $u_h$, we have
\begin{align}
 & \int_{T_{n-1}}^{T_{n}}\int_{\Omega}\cfrac{\partial(u_{h}-u_{H})}{\partial t}v+\int_{T_{n-1}}^{T_{n}}\int_{\Omega}\kappa\nabla(u_{h}-u_{H})\cdot\nabla v+\int_{\Omega}(u_{h}-u_{H})v|_{t=T_{n-1}^{+}}\nonumber \\
= & \int_{\Omega}\left(g_{h}^{(n)}-g_{H}^{(n)}\right)v(x,T_{n-1}^{+}),\;\forall v\in V_{H}^{(n)}.\label{eq:difference of u_h and u_H}
\end{align}
Therefore, taking $v=w-u_{H}$ and combining the equation (\ref{eq:V norm equality})
and (\ref{eq:difference of u_h and u_H}), we obtain
\begin{eqnarray*}
\|u_{h}-u_{H}\|_{V^{(n)}}^{2} & = & \int_{T_{n-1}}^{T_{n}}\int_{\Omega}\cfrac{\partial(u_{h}-u_{H})}{\partial t}(u_{h}-w)+\int_{\Omega}(u_{h}-u_{H})(u_{h}-w)|_{t=T_{n-1}^{+}}\\
 &  & +\int_{T_{n-1}}^{T_{n}}\int_{\Omega}\kappa\nabla(u_{h}-u_{H})\cdot\nabla(u_{h}-w) +\int_{\Omega}\left(g_{h}^{(n)}-g_{H}^{(n)}\right)(w-u_{H})|_{t=T_{n-1}^{+}}.
\end{eqnarray*}
Using integration by parts, we have
\begin{align*}
& \int_{T_{n-1}}^{T_{n}}\int_{\Omega}\cfrac{\partial(u_{h}-u_{H})}{\partial t}(u_{h}-w)+\int_{\Omega}(u_{h}-u_{H})(u_{h}-w)|_{t=T_{n-1}^{+}}\\
= & -\int_{T_{n-1}}^{T_{n}}\int_{\Omega}\cfrac{\partial(u_{h}-w)}{\partial t}(u_{h}-u_{H})+\int_{\Omega}(u_{h}-u_{H})(u_{h}-w)|_{t=T_{n}^{-}}.
\end{align*}
Thus,
\begin{eqnarray*}
\|u_{h}-u_{H}\|_{V^{(n)}}^{2}
& = & -\int_{T_{n-1}}^{T_{n}}\int_{\Omega}\cfrac{\partial(u_{h}-w)}{\partial t}(u_{h}-u_{H})
     +\int_{\Omega}(u_{h}-u_{H})(u_{h}-w)|_{t=T_{n}^{-}}\\
 &  & +\int_{T_{n-1}}^{T_{n}}\int_{\Omega}\kappa\nabla(u_{h}-u_{H})\cdot\nabla(u_{h}-w)
      +\int_{\Omega}\left(g_{h}^{(n)}-g_{H}^{(n)}\right)(u_{h}-u_{H})|_{t=T_{n-1}^{+}}\\
 &  & +\int_{\Omega}\left(g_{h}^{(n)}-g_{H}^{(n)}\right)(w-u_{h})|_{t=T_{n-1}^{+}}\\
 &\leq& C\|\cfrac{\partial(u_{h}-w)}{\partial t}\|_{L^{2}((T_{n-1},T_{n});H^{-1}(\kappa))}
 \|u_{h}-u_{H}\|_{L^{2}((T_{n-1},T_{n});\kappa)}\\
 &  & +\|(u_{h}-u_{H})(\cdot,T_{n}^{-})\|_{L^{2}(\Omega)}\|(u_{h}-w)(\cdot,T_{n}^{-})\|_{L^{2}(\Omega)}\\
 &  & +\|u_{h}-w\|_{L^{2}((T_{n-1},T_{n});\kappa)}\|u_{h}-u_{H}\|_{L^{2}((T_{n-1},T_{n});\kappa)}\\
 &  & +\|g_{h}^{(n)}-g_{H}^{(n)}\|_{L^{2}(\Omega)}(\|(u_{h}-u_{H})(\cdot,T_{n-1}^{+})\|_{L^{2}(\Omega)}\\
 &  & +\|(u_{h}-w)(\cdot,T_{n-1}^{+})\|_{L^{2}(\Omega)}).
\end{eqnarray*}
Using Young's inequality, we have
\[
\|u_{h}-u_{H}\|_{V^{(n)}}^{2} \leq
\cfrac{1}{2}\|u_{h}-u_{H}\|_{V^{(n)}}^{2} + 2\left( C\|u_{h}-w\|_{W^{(n)}}^{2}+\|g_{h}^{(n)}-g_{H}^{(n)}\|_{L^{2}(\Omega)}^{2}\right).
\]
and
\begin{align*}
\|g_{h}^{(n)}-g_{H}^{(n)}\|_{L^{2}(\Omega)}^{2} & =\begin{cases}
0 & \text{ for }n=1\\
\|u_{h}^{(n-1)}(\cdot,T_{n-1}^{-})-u_{H}^{(n-1)}(\cdot,T_{n-1}^{-})\|_{L^{2}(\Omega)}^{2} & \text{ for }n>1
\end{cases}\\
 & \leq\begin{cases}
0 & \text{ for }n=1\\
\|u_{h}-u_{H}\|_{V^{(n-1)}}^{2} & \text{ for }n>1
\end{cases}.
\end{align*}
Therefore, we proved the lemma.

\end{proof}

\bibliographystyle{plain}
\bibliography{references}

\end{document}